\documentclass{article}

\usepackage[T1]{fontenc}
\usepackage{enumerate, amsmath, amsfonts, amssymb, amsthm, dsfont, mathrsfs, wasysym, graphics, graphicx, xcolor, url, hyperref, hypcap, xargs, multicol, pdflscape, multirow, hvfloat, array, ae, aecompl, pifont, mathtools, a4wide, float, blkarray, overpic, nicefrac, stmaryrd, anyfontsize, yfonts}
\usepackage{CJKutf8}
\usepackage{xargs, bbm, enumerate, paralist}
\usepackage[shortlabels, inline]{enumitem}
\usepackage{pdflscape}
\usepackage[noabbrev,capitalise]{cleveref}
\usepackage[normalem]{ulem}
\usepackage{marginnote}
\usepackage{animate}

\hypersetup{colorlinks=true, citecolor=darkblue, linkcolor=darkblue}
\usepackage[all]{xy}
\usepackage{tikz}
\usepackage{tikz-cd}
\usetikzlibrary{trees, decorations, decorations.pathmorphing, decorations.markings, decorations.shapes, shapes, arrows, matrix, calc, fit, intersections, patterns, angles}
\graphicspath{{figures/}{figures/diagonals/}{figures/walks/}{figures/tubes/}{figures/blocks/}}
\makeatletter\def\input@path{{figures/}}\makeatother
\usepackage{caption}
\usepackage[export]{adjustbox}

\usepackage{paralist}
\usepackage{shuffle}

\DeclareFontEncoding{LY}{}{}
\DeclareFontSubstitution{LY}{yfrak}{m}{n}
\DeclareFontEncoding{LYG}{}{}
\DeclareFontSubstitution{LYG}{ygoth}{m}{n}
\DeclareFontFamily{LYG}{ygoth}{}
\DeclareFontShape{LYG}{ygoth}{m}{n}{<->ygoth}{}
\DeclareFontFamily{LY}{yfrak}{}
\DeclareFontShape{LY}{yfrak}{m}{n}{<->yfrak}{}
\DeclareFontFamily{LY}{ysmfrak}{}
\DeclareFontShape{LY}{ysmfrak}{m}{n}{<->ysmfrak}{}
\DeclareFontFamily{LY}{yswab}{}
\DeclareFontShape{LY}{yswab}{m}{n}{<->yswab}{}

\newtheorem{theorem}{Theorem}[section]
\newtheorem{corollary}[theorem]{Corollary}
\newtheorem{proposition}[theorem]{Proposition}
\newtheorem{lemma}[theorem]{Lemma}
\newtheorem{conjecture}[theorem]{Conjecture}
\newtheorem*{theorem*}{Theorem}

\theoremstyle{definition}
\newtheorem{definition}[theorem]{Definition}
\newtheorem{example}[theorem]{Example}
\newtheorem{remark}[theorem]{Remark}
\newtheorem{question}[theorem]{Problem}

\crefname{equation}{Equation}{Equations}

\newcommand{\R}{\mathbb{R}} 
\newcommand{\N}{\mathbb{N}} 

\renewcommand{\c}[1]{{\mathcal{#1}}} 
\renewcommand{\b}[1]{{\boldsymbol{#1}}} 
\newcommand{\scr}[1]{{\mathscr{#1}}} 
\newcommand{\go}[1]{{\textgoth{#1}}} 

\renewcommand{\epsilon}{\varepsilon} 

\newcommand{\ssm}{\smallsetminus} 
\newcommand{\one}{{1\!\!1}} 
\newcommandx{\ones}[1][1=n]{\one_{#1}} 

\DeclareMathOperator{\conv}{conv} 
\DeclareMathOperator{\st}{st} 

\newcommand{\ie}{\textit{i.e.}~} 
\newcommand{\eg}{\textit{e.g.}~} 
\definecolor{darkblue}{rgb}{0,0,0.7} 
\definecolor{green}{RGB}{57,181,74} 
\definecolor{violet}{RGB}{147,39,143} 
\newcommand{\darkblue}{\color{darkblue}} 
\newcommand{\defn}[1]{\textsl{\darkblue #1}} 

\makeatletter
\def\part{\@startsection{part}{1}%
\z@{.7\linespacing\@plus\linespacing}{.8\linespacing}%
{\LARGE\sffamily\centering}}
\makeatother

\newcommand{\polytope}[1]{\mathsf{#1}}
\newcommand{\polytopeP}{\mathsf{P}}

\newcommand{\polytopeQ}{\mathsf{Q}}
\newcommand{\polytopeF}{\mathsf{F}}

\newcommand{\Cube}[1][n]{\square_{#1}}
\newcommandx{\PivotPolytope}[2][1=d,2=\t]{\polytope{\Pi}_{#2}^{#1}}
\newcommandx{\PP}{\polytope{\Pi}}

\newcommand{\permuto}[1][n]{\polytope{\Pi}_{#1}}
\newcommandx{\HypSimpl}[2][1=n,2=k]{\polytope{\Delta}(#1,#2)}
\newcommand{\HypSimplTwo}[1][n]{\HypSimpl[#1][2]}
\newcommand{\supp}[1][\b v]{\text{s}(#1)}

\newcommandx{\MPP}[2][1=\polytopeP,2=\b c]{\polytope{M}_{#2}(#1)}
\newcommandx{\MPPHypSimpl}[2][1=n,2=k]{\polytope{M}(#1,#2)}
\newcommand{\MPPHypSimplTwo}[1][n]{\polytope{M}(#1,2)}

\newcommandx{\gZono}[1][1=G]{\mathsf{Z}_{#1}}

\newcommandx{\Asso}[2][1=n,2={}]{\mathsf{Asso}^{#2}(#1)} 
\newcommandx{\dZono}[1][1=\b{h}]{\mathsf{D}_{#1}} 
\newcommand{\simplex}{\polytope{\Delta}} 

\newcommandx{\Fan}[1][1=F]{\mathcal{#1}} 
\newcommandx{\nestedFan}[1][1=\quiver]{\mathcal{F}(#1)} 

\newcommandx{\ray}[1][1=r]{\b{#1}} 
\newcommandx{\rays}[1][1=R]{\b{#1}} 
\newcommandx{\Perm}[1][1=n]{\polytope{Perm}_{#1}}

\newcommandx{\gArr}[1][1=G]{\mathcal{A}_{#1}} 
\newcommandx{\gFan}[1][1=G]{\Fan_{#1}} 
\newcommandx{\gFanO}[1][1=G]{\widehat{\Fan}_{#1}} 
\newcommandx{\cc}[1][1=G]{\mathbb{K}_{#1}} 

\newcommandx{\braid}[1][1=n]{\mathcal{B}_{#1}} 
\newcommandx{\sbraid}[1][1=n]{\widehat{\mathcal{B}}_{#1}} 

\newcommandx{\coefficient}[3][1={\b{s}}, 2=\b{r}, 3=\b{r}']{\alpha_{#2,#3}(#1)} 
\newcommandx{\virtualPolytopes}[1][1=d]{\mathbb{V}^{#1}} 
\newcommandx{\VDP}[1][1=n]{\mathbb{VDP}^{#1}} 
\newcommandx{\CVDP}[1][1=n]{\overrightarrow{\mathbb{VDP}}^{#1}} 
\newcommand{\VD}[1][1=n]{\mathbb{VD}} 
\newcommandx{\opcone}[1][1={\mu,\omega}]{\polytope{C}_{#1}}
\newcommandx{\orcone}[1][1={\omega}]{\polytope{C}_{#1}}

\newcommandx{\graphG}[1][1=G]{#1} 
\newcommandx{\hypergraph}[1][1=H]{\graphG[#1]} 
\newcommandx{\tube}[1][1=t]{\mathsf{#1}} 
\newcommandx{\tubes}[1][1=\graphG]{\building#1} 
\newcommandx{\tubing}[1][1=T]{\mathsf{#1}} 

\newcommand{\building}{\mathcal{B}} 
\newcommandx{\nested}[1][1=N]{\mathcal{#1}} 

\newcommand{\vmin}{\b v_{\min}}
\newcommand{\vmax}{\b v_{\max}}
\newcommand{\slope}{\tau_{\b \omega}}
\newcommandx{\enhancedStep}[3][1=i, 2=j, 3=a]{#1 \xrightarrow{#3} #2}
\newcommandx{\step}[2][1=i, 2=j]{#1 \rightarrow #2}
\newcommandx{\enhancedStepx}[3][1=x, 2=y, 3=z]{#1 \xrightarrow{#3} #2}
\newcommandx{\enhancedStepZ}[3][1=x, 2=y, 3=Z]{#1 \xrightarrow{#3} #2}

\renewcommand\t{\b t}%

\newcommandx{\ConstrainedMALIntrinsic}[2][1=n,2=d]{\go M_{#1,#2}}

\newcommandx{\AssGraph}[2][1=n,2=3]{\mathit{Asso}_{#1}^{#2}}

\newcommandx{\Fpolytope}[3][1=d,2=A,3=\b t]{\polytopeP_{#1}^f\left(#2,#3 \right)}
\newcommandx{\Bpolytope}[3][1=d,2=A,3=\b t]{\polytopeP_{#1}^b \left(#2,#3 \right)}

\newcommandx{\FibPol}[3][1=\polytopeP,2=\polytopeQ,3=\pi]{\polytope{\Sigma}_{#3}(#1,#2)}

\newcommandx{\FibPolCyc}[2][1=d,2=\t]{\polytope{\Sigma}^{#1}_{2}(#2)}
\newcommandx{\PolProj}[3][1=\polytopeP,2=\polytopeQ,3=\pi]{#3~:~#1\to #2}

\newcommandx{\HOmega}[3][1=\kappa,2=\t,3=d]{\Omega_{\kappa}^d(\t)}
\newcommandx{\rHOmega}[3][1=\kappa,2=\t,3=d]{\overline{\Omega}_{\kappa}^d(\t)}
\newcommandx{\Ppolytope}[3][1=d,2=T,3=\t]{\polytopeQ^+_{#1}(#2,#3)}
\newcommandx{\Npolytope}[3][1=d,2=T,3=\t]{\polytopeQ^-_{#1}(#2,#3)}

\newcommandx{\gVdM}[3][1=n,2=k,3=\b \lambda]{\text{VdM}_{#1,#2}(#3)}
\newcommandx{\VdM}[2][1=n,2=\b\lambda]{\text{VdM}_{#1}(#2)}

\newcommand{\inner}[1]{\left<#1\right>}

\usepackage{todonotes}

\AtEndDocument{%
  \par
  \medskip
  \begin{tabular}{@{}l@{}}%
    Germain Poullot, \textsc{Universit\"at Osnabr\"uck}\\
    \textit{E-mail}: \texttt{germain.poullot@uni-osnabrueck.de}
  \end{tabular}}

\title{Vertices of the monotone path polytopes of hypersimplices}
\author{Germain Poullot}
\date{}

\begin{document}

\maketitle

\begin{abstract}
The monotone path polytope of a polytope $\polytopeP$ encapsulates the combinatorial behavior of the shadow vertex rule (a pivot rule used in linear programming) on $\polytopeP$.
Computing monotone path polytopes is the entry door to the larger subject of fiber polytopes, for which explicitly computing examples remains a challenge.
We first give a detailed presentation on how to construct monotone path polytopes.

Monotone path polytopes of cubes and simplices have been known since the seminal article of Billera and Sturmfels \cite{BilleraSturmfels-FiberPolytope}.
We extend these results to hypersimplices by linking this problem to the combinatorics of lattice paths.
Indeed, we give a combinatorial model which describes the vertices of the monotone path polytope of the hypersimplex $\HypSimplTwo$ (for any generic direction).
With this model, we give a precise count of these vertices, and furthermore count the number of coherent monotone paths on $\HypSimplTwo$ according to their lengths.

We prove that some of the results obtained also hold for hypersimplices $\HypSimpl$ for $k\geq 2$.
\end{abstract}

\section*{Acknowledgment}

The author wants to thank Jes\'us De Loera for the very interesting discussion we had in Paris in December 2022, together with Eva Philippe.
\Cref{ssec:CountingByLengths} follows from the questions we discussed there.
Besides, a lot of gratitude goes to Arnau Padrol and Vincent Pilaud for the numerous encouragements to work on the subject and the re-reading that ensued, during the thesis of the author,
and to Martina Juhnke for her precious advises and corrections during the transformation of this extract of a thesis into an article.
Special thanks go to Alex Black for asking the author the original question that this article solves during the conference in Bielefeld in September 2022.

\tableofcontents

\begin{small}
\listoffigures
\end{small}


\section{Introduction}

After their introduction by Billera and Sturmfels \cite{BilleraSturmfels-FiberPolytope},  fiber polytopes received a lot of attention.
In particular, the fiber polytope for the projection of a polytope $\polytopeP$ onto a segment encapsulates the combinatorics of (coherent) monotone paths on $\polytopeP$:
it is the \defn{monotone path polytope} of $\polytopeP$ \cite{athanasiadis1999piles,athanasiadis2000monotone,blanchard2020length}.
Its vertices are in bijection with the monotone paths that can be followed by a shadow vertex rule.
As such, it links the world of linear optimization to the world of triangulations.

This, and the fact that monotone path polytopes stand among the easiest fiber polytopes to compute, have motivated numerous studies on the subject.
Especially, the monotone path polytope of a simplex is a cube \cite{BilleraSturmfels-FiberPolytope}, the one of a cube is a permutahedron \cite[Example 9.8]{BilleraSturmfels-FiberPolytope,Ziegler-polytopes}, the one of a cyclic polytope is a cyclic zonotope \cite{Athanasiadis_2000}, the one of a cross-polytope is the signohedron \cite{BlackDeLoera2021monotone}, and the one of an $S$-hypersimplex is a permutahedron \cite{ManeckeSanyalSo2019shypersimplices}.

However, monotone path polytopes of the hypersimplices have not yet been fully explored.
The \defn{$(n,k)$-hypersimplex $\HypSimpl$} is 
defined as the convex hull of all $(0,1)$-vectors with $k$ ones and $n-k$ zeros \cite[Example 0.11]{Ziegler-polytopes}.
Recently, Olarte and Santos \cite{OlarteSantos-HypersimplicialSubdivisions} introduced \emph{$k^\text{th}$-hypersecondary polytopes} which is fiber polytope for the projection of the hypersimplex $\HypSimpl$ onto (the $k^\text{th}$ Minkowski sum of) a point configuration.
Although not discussed in \cite{OlarteSantos-HypersimplicialSubdivisions}, when this point configuration is contained into a line, then the hypersecondary polytope is equal to the monotone path polytope of the hypersimplex $\HypSimpl$.
Besides, in \cite[Section 8]{BlackSanyal-FlagPolymatroids}, the authors consider monotone path polytopes of hypersimplices (seen as matroid polytopes) for very specific directions $\b e_S := \sum_{i\in S} \b e_i$ for $S\subseteq [n]$, but the case of generic directions was left wide open.

\vspace{0.15cm}

In the present paper, we begin with a general introduction to monotone path polytopes (\Cref{ssec:ConstructionOfMPP}), and then examine monotone path polytopes of hypersimplices (\Cref{ssec:CoherentPathsOnHypSimpl}), especially their vertices.
Even though we borrow the definition of monotone path polytopes from the realm of fiber polytopes, we do not aim at giving a presentation of fiber polytopes but rather of monotone path polytopes for themselves.
Therefore, the reader unfamiliar with fiber polytopes should feel at ease (after \Cref{def:MPP_via_fiber}), and if he or she wants an instructive and illustrated presentation of fiber polytopes, we advise to look at \cite[Chapter 9]{Ziegler-polytopes}, or at a more in-depth explanation in \cite[Section 2]{Athanasiadis_2000} and \cite[Chapter 9.1]{DeLoeraRambauSantos-Triangulations}, and at the original article \cite{BilleraSturmfels-FiberPolytope}.

We give a necessary criterion for a monotone path on $\HypSimpl$ to appear as a vertex of its monotone path polytope (\Cref{ssec:EnhancedStepCriterion}), and link this criterion to the case of a non-generic direction (\Cref{ssec:NonGenericDirection}).
We prove that this criterion is furthermore sufficient in the case of the second hypersimplex $\HypSimplTwo$, by constructing a combinatorial model of lattice paths to study it (\Cref{ssec:TowardsLatticePaths}), and by proving an induction process on this model (\Cref{ssec:InductionProcess}).
We finish by giving the exact count of the vertices of the monotone path polytope of $\HypSimplTwo$ (\Cref{ssec:CountingTotal,thm:NumberVertsMPPHypSimplTwo}): by detailing our formula, we can even count the number of coherent monotone paths on $\HypSimplTwo$ according to their lengths (\Cref{ssec:CountingByLengths}).

\section{Preliminaries on monotone paths polytopes}\label{ssec:MonotonePathPolytope}

\subsection{Explicit construction(s) of monotone path polytopes}\label{ssec:ConstructionOfMPP}

In general, fiber polytopes are, by construction, complicated to compute.
As a simple case, fiber polytopes for projections onto a point are trivial: the fiber polytope of a polytope $\polytopeP$ onto a point equals $\polytopeP$ itself.
Hence, among the first cases one would want to investigate are the fiber polytopes associated to projections onto a 1-dimensional polytope, \ie a segment.
We state here the formal definition coming from fiber polytopes for reference, but we will never use it, preferring instead the reformulation of \Cref{thm:MPPSumSectionAtVertices,thm:FaceLatticeMPP}: the non-expert reader does not need the following definition to understand the results thereafter.

In this article, a (bounded) \defn{linear program $(\polytopeP, \b c)$} is a couple formed by a polytope $\polytopeP\subset \R^d$ and a vector $\b c\in \R^d$ that shall be thought as the direction along which we want to optimize.
Except if we say otherwise, we always take a \defn{generic direction} $\b c$, that is to say $\inner{\b c, \b v - \b u} \ne 0$ for every edge $\b u\b v$ of $\polytopeP$.
This ensures that the scalar product against $\b c$ is minimized (respectively maximized) by a unique vertex of $\polytopeP$: \defn{$\b v_{\min} := \text{argmin}\{\inner{\b c, \b x}~;~\b x\in \polytopeP\}$} and \defn{$\b v_{\max} := \text{argmax}\{\inner{\b c, \b x}~;~\b x\in \polytopeP\}$}.

\begin{definition}\label{def:MPP_via_fiber}
For a linear program $(\polytopeP,\b c)$, the \defn{monotone path polytope} $\MPP$ is the fiber polytope for the projection $\pi_{\b c}:\b x\mapsto\inner{\b x,\b c}$.
Denoting the image segment $\polytopeQ = \pi_{\b c}(\polytopeP) = \{\inner{\b x,\b c} ~;~ \b x\in \polytopeP\}$, one has:
$$\MPP := \left\{\int_{\polytopeQ} \gamma(x)\mathrm{d}x ~~;~~ \gamma \text{ section of } \pi_{\b c}\right\}$$
\end{definition}

The polytope $\MPP$ has dimension $\dim(\polytopeP)-1$ but is embedded in $\R^{\dim(\polytopeP)}$.

The monotone path polytope, though arising from a fiber polytope point of view, is deeply linked to linear programming.
The following theorem formalizes this link.
We will then detail some examples, in order to clarify the definitions.

\begin{definition}\label{def:CellularString}
For a linear program $(\polytopeP,\b c)$, a \defn{cellular string} is a sequence $\sigma = (\polytopeF_1,\dots,\polytopeF_k)$ of faces of $\polytopeP$ (of dimensions at least 1) such that $\min(\polytopeF_1) = \b v_{\min}$, $\max(\polytopeF_k) = \b v_{\max}$, and for all $i\in [k-1]$, $\max(\polytopeF_i) = \min(\polytopeF_{i+1})$, where minima and maxima are taken with respect to the scalar product against $\b c$.
Cellular strings are ordered by containment of their union to form a partially ordered set (which is known to be a lattice).
The \defn{finest} cellular strings are the smallest (non-empty) cellular strings for this order, \ie sequences of edges.
\end{definition}

\begin{definition}\label{def:CoherentCellularString}
For a linear program $(\polytopeP, \b c)$ and a secondary direction $\b\omega\in \R^d$ (linearly independent to $\b c$), one can consider the polygon \defn{$\polytopeP_{\b c,\b \omega}$} obtained by projecting $\polytopeP$ onto the plane spanned by $(\b c, \b \omega)$:
$$\polytopeP_{\b c, \b \omega} := \left\{\bigl(\inner{\b x,\b c},~\inner{\b x,\b \omega}\bigr) ~;~ \b x\in \polytopeP\right\}$$

A proper face (vertex or edge) $\polytope{G}$ of $\polytopeP_{\b c, \b \omega}$ is an \defn{upper face} if it has an outer normal vector with positive second coordinate\footnote{Some definitions in the literature use lower faces, we take upper faces to ease drawings and notations.}, equivalently if $(x_1, x_2) + (0, \varepsilon) \notin \polytopeP_{\b c, \b\omega}$ for all $(x_1, x_2) \in \polytope{G}$, and $\varepsilon > 0$.

A cellular string $\sigma$ is \defn{coherent} if there exists $\b\omega\in \R^d$ such that $\sigma$ is the family of pre-images by~$\pi_{\b c}$ of the lower faces of $\polytopeP_{\b c, \b \omega}$.
In this case, such an $\b\omega$ is said to \defn{capture} the cellular string $\sigma$.
\end{definition}

\begin{theorem}\label{thm:FaceLatticeMPP}\emph{(\cite[Theorem 2.1]{BilleraSturmfels-FiberPolytope}).}
The face lattice of $\MPP$ is the lattice of coherent cellular strings on $\polytopeP$.
\end{theorem}

\Cref{def:CellularString,def:CoherentCellularString} may be hard to parse: we now look at the vertices of the monotone path polytopes $\MPP$ and expose this construction more clearly.
\Cref{fig:AnimatedMPPSimplex} gives an illustration.

Fix a polytope $\polytopeP\subset\R^d$, and a generic direction $\b c\in \R^d$.
Each edge $\b u\b v$ of $\polytopeP$ is oriented by $\b c$, from $\b u$ to $\b v$ if and only if $\inner{\b c, \b u} < \inner{\b c, \b v}$.
The finest cellular strings are \defn{monotone paths} on $\polytopeP$, that is to say a sequence of vertices $\b v_0, \b v_1, \dots, \b v_r$ such that $\b v_0 = \b v_{\min}$, $\b v_r = \b v_{\max}$ and $\b v_i\b v_{i+1}$ is an edge of $\polytopeP$ directed by $\b c$ from $\b v_i$ to $\b v_{i+1}$.


\begin{figure}[t]
    \centering
\animategraphics[autoplay,loop,width=\textwidth,controls]{12}{Figures/MPPSimplexTurning/SimplexTurningV2-}{0}{63}
    \caption[\textcolor{cyan}{Animated} Construction of the normal fan of the monotone path polytope of $\simplex_3$]{Animated construction of the normal fan of the monotone path polytope of the 3-dimensional simplex.
    For each $\b\omega\in \R^3$ orthogonal to $\b c$, we project $\simplex_3$ onto the plane spanned by $(\b c,\b\omega)$ (Right), and record the corresponding coherent monotone path (Left).\linebreak
    \emph{(Animated figures obviously do not display on paper, please use a PDF viewer (like Adobe Acrobat Reader), or go on my personal website, or ask by email.)}\vspace{-0.3cm}}
    \label{fig:AnimatedMPPSimplex}
\end{figure}




Now, consider another direction $\b \omega\in \R^d$, linearly independent of $\b c$, and project $\polytopeP$ onto the plane spanned by~$(\b c, \b \omega)$.
In \Cref{fig:AnimatedMPPSimplex}, we take a tetrahedron and a direction $\b c$, and then scan through all possible $\b \omega$ (it is enough to scan only $\b\omega$ with $\inner{\b \omega, \b c} = 0$).
The tetrahedron is naturally projected onto the plane $(\b c, \b \omega)$ by looking at the boundary of the drawing (\ie the 3 or 4 outside edges).
This polygon has two paths from its minimal (leftmost) vertex to its maximal (rightmost) vertex: a lower one and a upper one.
When there exists $\b\omega$ such that a given monotone path $\c L$ is projected onto the upper path (and no $2$-face of $\polytopeP$ projects onto it), then $\c L$ is coherent.

The vertices of the monotone path polytope $\MPP$ are not in bijection to all the monotone paths on $\polytopeP$, but only to the coherent ones.
The faces of higher dimension are obtained following the same ideas.
In \Cref{fig:AnimatedMPPSimplex}, we record on the left the (coherent) monotone path obtained for each choice of $\b\omega$ on the chosen tetrahedron, in this case, all monotone paths are coherent.

We now present four ways to visualize the monotone path polytope: the two firsts focus on its normal fan, while the two lasts allow for an explicit computation of the vertices.
For computing the monotone path polytopes of hypersimplices, we will use \Cref{def:CoherentCellularString}, but we believe the reader may find it pleasant to see explicit constructions of the monotone path polytope in general.

\paragraph{Exploring the space of $\b\omega$, and projection of the normal fan of $\polytopeP$}
First of all, \Cref{def:CoherentCellularString} invites us to focus on the space of all $\b \omega$, and partition it depending on the coherent cellular strings they yield.
Precisely, to a cellular string $\sigma$ we associate $\c N(\sigma) = \{\b\omega~;~\b\omega \text{ captures } \sigma\}$.
Then $\c N(\sigma)$ is a polyhedral cone by linearity (in $\b \omega$) of the projection from $\polytopeP$ onto $\polytopeP_{\b c, \b \omega}$, and $\c N = \bigl(\c N(\sigma)\bigr)_{\sigma}$ is a fan.
This fan is the normal fan of $\MPP$.
Hence, one can run through all possible $\b\omega\in\R^d$, orthogonal to $\b c$ (as all $\b\omega + \lambda\b c$ capture the same cellular string for any $\lambda\in \R$), to draw the normal fan of $\MPP$, see \Cref{fig:AnimatedMPPSimplex} for the construction. 
A given $\b \omega$ captures the coherent path $(\b v_0, \dots, \b v_r)$ such that $\b v_0 = \b v_{\min}$, $\b v_r = \b v_{\max}$ and $\b v_i\b v_{i+1}$ is the edge of $\polytopeP$ with $\inner{\b v_i, \b c} < \inner{\b v_{i+1}, \b c}$ satisfying that $\frac{\inner
{\b v_{i+1} - \b v_i, \b \omega}}{\inner
{\b v_{i+1} - \b v_i, \b c}}$ is the unique maximizer of $\frac{\inner
{\b v_j - \b v_i, \b \omega}}{\inner
{\b v_j - \b v_i, \b c}}$ for $\b v_i\b v_j$ and edge of $\polytopeP$ with $\inner{\b v_i, \b c} < \inner{\b v_j, \b c}$.

This construction shows an important property: for a fixed $\b\omega$, all $\b \omega+\lambda\b c$ for $\lambda\in \R$ capture the same cellular string.
Consequently, one can obtain the normal fan of $\MPP$ by \emph{projecting} the normal fan of $\polytopeP$:
to each normal cone $\polytope{C}\in \c N_{\polytopeP}$, associate its projection along $\b c$, namely $\polytope{C}_{\perp} := \{\b x - \frac{\inner{\b x,\b c}}{\inner{\b c,\b c}}\b c ; \b x\in \polytope{C}\}$.
The common refinement of $\bigl(\polytope{C}_{\perp} ; \polytope{C}\in\c N_{\polytopeP}\bigr)$ is the normal fan\footnote{This construction embeds the fan $\c N_{\MPP}$ directly into the hyperplane $\b c^{\perp}$, instead of embedding it in $\R^{\dim(\polytopeP)}$.} of $\MPP$.

\begin{figure}
    \centering
    \includegraphics{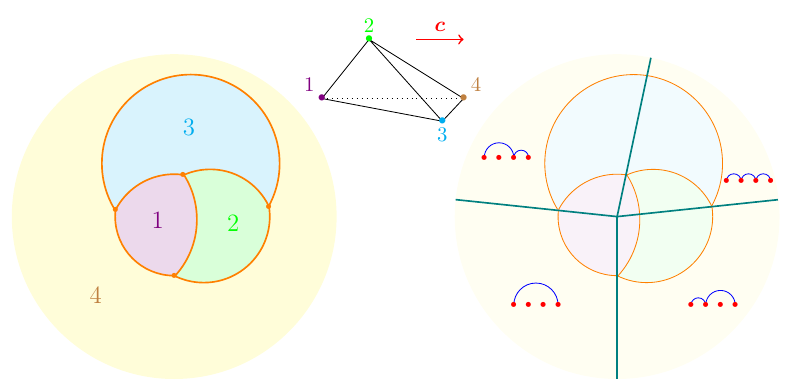}
    \caption[Normal fan of $\MPP$ of a tetrahedron through stereographic projection]{(Top) For reference, the tetrahedron $\polytopeP = \simplex_3$, and the direction $\b c$ from \Cref{fig:AnimatedMPPSimplex}.
    (Left) The stereographic projection of the normal fan of $\polytopeP$, each colored region correspond to (the normal cone of) a vertex of $\polytopeP$.
    (Right) Two rays give rise to the same coherent monotone path if and only they intersect the same colored regions, we draw the resulting fan, labeled accordingly.\vspace{-0.4cm}}
    \label{fig:StereographicMPP}
\end{figure}

\paragraph{Line bundle and stereographic projection}
A second way to visualize the combinatorics of the monotone path polytope $\MPP$ is to imagine the lines $\ell_{\b \omega} := (\b\omega + \lambda\b c ~;~ \lambda\in\R)$, and consider the \emph{line bundle} $(\ell_{\b\omega} ~;~ \b\omega\in\R^d)$.
Each line $\ell_{\b\omega}$ intersects the normal fan $\c N_{\polytopeP}$ of the polytope $\polytopeP$, and the cones it intersects describe the cellular string that $\b\omega$ captures: if $\ell_{\b\omega}$ intersects only maximal cones and cones of co-dimension 1, then it captures a coherent monotone path (which is the case for almost all $\b\omega$).
Looking at which maximal cones of $\c N_{\polytopeP}$ are intersected by $\ell_{\b\omega}$ yields the list of vertices forming the associated coherent monotone path;
whereas looking at which co-dimension~1 cones are intersected by $\ell_{\b\omega}$ yields the list of edges forming the associated coherent monotone path.

To visualize this easily (especially if $\dim\polytopeP = 3$), one can use the stereographic projection $\st_{\b c} : \R^d\to\R^{d-1}$ that maps the apex $\frac{\b c}{|\!|\b c|\!|}$ to infinity, see \Cref{fig:StereographicMPP}.
The normal fan $\c N_\polytopeP$ projects onto a subdivision $\st_{\b c}(\c N_\polytopeP)$ of $\R^{d-1}$ by spherical cap (\ie arcs of circles if $\dim\polytopeP = 3$).
Besides, the counterpart on the sphere of $\ell_{\b \omega}$, is the arc $\alpha_{\b\omega} = \bigl(\frac{\b \omega + \lambda\b c}{|\!|\b \omega + \lambda\b c|\!|} ~;~  \lambda\in\R\bigr)$.
This is an arc of a great circle containing the apex $\frac{\b c}{|\!|\b c|\!|}$ and its antipodal point:
thus $\st_{\b c}(\alpha_{\b\omega})$ is ray (from $\b 0\in \R^{d-1}$ to infinity).
The cells of $\st_{\b c}(\c N_\polytopeP)$ that the ray $\st_{\b c}(\alpha_{\b\omega})$ intersects are the cones of $\c N_\polytopeP$ that $\ell_{\b\omega}$ intersects, and hence describe the coherent path that $\b \omega$ captures: by looking how all rays of $\R^{d-1}$ intersect $\st_{\b c}(\c N_\polytopeP)$, we directly get a drawing of the normal fan of $\MPP$ in $\R^{d-1}$, see \Cref{fig:StereographicMPP} (Right).

This point of view can come in handy when one wants to vary $\b c$.
Indeed, varying $\b c$ amounts to varying the apex of the stereographic projection, \ie to ``roll'' the projection $\st_{\b c}(\c N_\polytopeP)$ inside $\R^{d-1}$.
This ``rolling'' is hard to describe, but at least, the rays we want to intersect it with remain fixed.

\paragraph{Convex hull of (explicit) points}
A third way to construct the monotone path polytope $\MPP$ is to use the following formula from \cite[Theorem 5.3]{BilleraSturmfels-FiberPolytope}. 
Let $V(\polytopeP) = \{\b v_1, \dots, \b v_n\}$ with $\inner{\b v_i, \b c} \leq \inner{\b v_j, \b c}$ for $i\leq j$.
For a monotone path $\c L = (\b v_{i_1}, \dots, \b v_{i_r})$ on $\polytopeP$ (with $i_1 = 1$ and $i_r = n$), denote 
$$\psi(\c L) = \sum_{j = 1}^r \frac{\inner{\b v_{i_j} - \b v_{i_{j-1}}, \b c}}{2\inner{\b v_n - \b v_1, \b c}}\, (\b v_{i_j} + \b v_{i_{j-1}})$$
Then $\MPP = \conv\bigl(\psi(\c L) ~;~ \text{for } \c L \text{ monotone path on }\polytopeP\bigr)$.
A point $\psi(\c L)$ is a vertex of $\MPP$ if and only if $\c L$ is a coherent monotone path.

For the case of the simplex, all monotone paths are coherent, so a figure would not be very enlightening.
We picture a better example in \Cref{fig:MPP42}: (Left) is drawn a 3-dimensional polytope, (Right) its monotone path polytope, obtained via the above formula.
The two red crosses correspond to non-coherent monotone paths $\c L$, for which the point $\psi(\c L)$ lie inside $\MPP$.

\paragraph{Minkowski sum of sections}
A forth way to visualize monotone path polytopes is to use \cite[Theorem 1.5]{BilleraSturmfels-FiberPolytope} which provides a re-writing of the integral of \Cref{def:MPP_via_fiber} as a finite Minkowski sum.
This sum is constructed as follows.
We begin by sorting the vertices of $\polytopeP$ according to their scalar product against $\b c$: $V(\polytopeP) = \{\b v_1,\dots, \b v_n\}$ with $\inner{\b v_i,\b c} \leq \inner{\b v_{i+1},\b c}$.
The segment $\polytopeQ = \pi_{\b c}(\polytopeP)$ is cut out by the projection into sub-segments $\polytope{C}_i := [q_i,q_{i+1}]$ with $q_i = \inner{\b v_i,\b c}$,
and the barycenter (\ie middle) of $\polytope{C}_i$ is trivially $b_i = \frac{q_i+q_{i+1}}{2}$.
The monotone path polytope $\MPP$ is normally equivalent to the Minkowski sum of sections $\sum_{i=1}^n \pi_{\b c}^{-1}(b_i)$.

Though exact, this construction is a bit unhandy.
Yet, as we will prove in \Cref{thm:MPPSumSectionAtVertices}, one can forget about centers, as $\MPP$ is normally equivalent to $\sum_{i=2}^{n-1} \pi_{\b c}^{-1}(q_i)$.
This gives beautiful pictures, see \Cref{fig:MPPSumSectionSimplex} for the case of the tetrahedron.

Note that, between the figures, a slight change of perspective happened.
The fans constructed in \Cref{fig:AnimatedMPPSimplex} (Left) and \Cref{fig:StereographicMPP} (Right) are the same, and are the normal fan of the $\MPP$ appearing in \Cref{fig:MPPSumSectionSimplex} (Right), even thought a right angle \emph{seems} to appear on the latter but not on the firsts.

\begin{figure}
    \centering
    \includegraphics{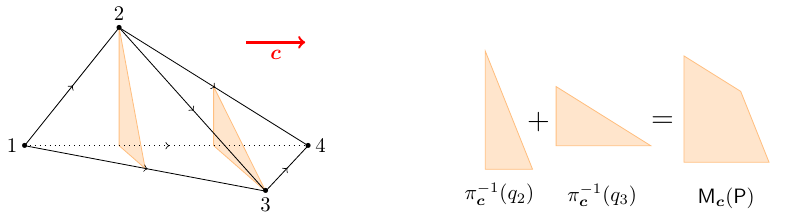}
    \caption[Construction of $\MPP$ as a sum of sections for the tetrahedron]{The construction of $\MPP$ as a sum of sections for the tetrahedron $\polytopeP = \simplex_3$.
    Each section is orthogonal to $\b c$ and contains a vertex (except for $\b v_{\min}$ and $\b v_{\max}$).}
    \label{fig:MPPSumSectionSimplex}
\end{figure}

\begin{theorem}[Folklore]\label{thm:MPPSumSectionAtVertices}
For a linear program $(\polytopeP, \b c)$ with $\b c$ generic, let $V(\polytopeP) = \{\b v_1,\dots, \b v_n\}$, and $q_i = \inner{\b v_i,\b c}$ with $q_1 \leq \dots \leq q_n$.
The monotone path polytope $\MPP$ is normally equivalent to the Minkowski sum of sections $\sum_{i=2}^{n-1} \{\b x\in \polytopeP ~;~ \inner{\b x,\b c} = q_i\}$.
\end{theorem}

\begin{remark}
As for the three previous constructions, this theorem is not new, but is folklore.
We give a self-contained proof: the reader shall rather think of it as an exercise on Cayley polytopes.
\end{remark}

\begin{proof}[Proof of \Cref{thm:MPPSumSectionAtVertices}]
Let $\polytopeP_{i} = \{\b x\in \polytopeP ~;~ \inner{\b x, \b c} = q_i\}$, and $\polytopeP_{i, i+1} = \{\b x\in \polytopeP ~;~ q_i\leq \inner{\b x, \b c} \leq q_{i+1}\}$.
Clearly: $\polytopeP_{i, i+1} = \conv\bigl(\polytopeP_i, \,\polytopeP_{i+1}\bigr)$.
In particular: $\{\b x\in \polytopeP ~;~ \inner{\b x, \b c} = \frac{q_i+q_{i+1}}{2}\} = \frac{1}{2}\polytopeP_i + \frac{1}{2}\polytopeP_{i+1}$.
By \cite[Theorem 1.5]{BilleraSturmfels-FiberPolytope}, the polytope $\MPP$ is normally equivalent to $\sum_i \{\b x\in \polytopeP ~;~ \inner{\b x, \b c} = \frac{q_i+q_{i+1}}{2}\}$.
Hence, $\MPP$ is normally equivalent to $\sum_i\Bigl(\frac{1}{2}\polytopeP_i+\frac{1}{2}\polytopeP_{i+1}\Bigr) = \polytopeP_1 + \polytopeP_n + \sum_i \polytopeP_i$.
As $\polytopeP_0 = \{\b v_1\}$ and $\polytopeP_n = \{\b v_n\}$ are points, they only amount to translation of the whole sum and can be removed without changing the normal equivalence.
\end{proof}

\subsection{Coherent paths on hypersimplices}\label{ssec:CoherentPathsOnHypSimpl}

The rest of this section is devoted to monotone path polytopes of hypersimplices, and especially hypersimplices for $k = 2$.
Before presenting new results on this subject, we shortly recall two former results from Billera and Sturmfels \cite[end of Section 5]{BilleraSturmfels-FiberPolytope}.

\begin{theorem}[\cite{BilleraSturmfels-FiberPolytope}]\label{thm:BS92-Simplex}
For any simplex $\polytope{\Delta}$ on $n+1$ vertices, and any generic direction $\b c$, the monotone path polytope $\MPP[\polytope{\Delta}]$ is (isomorphic to) a cube of dimension $n-1$.
\end{theorem}

\begin{theorem}[\cite{BilleraSturmfels-FiberPolytope}]
For the standard cube $\Cube = [0,1]^n$ of dimension $n$, and the direction $\b c = (1,\dots,1)$, the monotone path polytope $\MPP[\Cube]$ is (a dilation of) the permutahedron $\permuto$ of dimension $n-1$: $\permuto = \conv\Bigl\{\bigl(\sigma(1), \dots, \sigma(n)\bigr) ~;~ \sigma\in S_n\Bigr\}$.
\end{theorem}

These two results motivate the study of the monotone path polytopes of hypersimplices.
Indeed, hypersimplices are a generalization of simplices, and arise as sections of the standard cube.

\begin{definition}
For $n\geq 2$, $k\in [n]$, the \defn{$(n,k)$-hypersimplex} is $\HypSimpl = \bigl\{\b x\in[0,1]^n ~;~ \sum_i x_i = k\bigr\}$.
It is the section of the standard cube $\Cube = [0,1]^n$ by the hyperplane $\left\{\b x\in \R^n ~;~\inner{\b x, (1, \dots, 1)} = k\right\}$.

The vertices of $\HypSimpl$ are exactly its $(0,1)$-coordinate elements: the $(0,1)$-vectors with $k$ ones and $n-k$ zeros.
We denote the \defn{support} of a vertex $\b v\in V(\HypSimpl)$ by $\supp := \{i ~;~ v_i = 1\}$.
Two vertices $\b u,\b v\in V(\HypSimpl)$ share an edge when $|\supp[\b u]\cap\supp| = k-1$, \ie to obtain $\b v$ from $\b u$, flip a zero to a one and a one to a zero.
\end{definition}

Note that the hypersimplices $\HypSimpl[n][1]$ and $\HypSimpl[n][n-1]$ are simplices: in this sense, hypersimplices are a generalization of simplices.
As sections of the standard cube, they also generalize the cube.

We consider the linear program $(\HypSimpl,\b c)$ where $\b c\in \R^n$ is generic with respect to $\HypSimpl$.
A direction $\b c$ is generic for the hypersimplex when $c_i\ne c_j$ for all $i\ne j$, as each edge has direction $\b e_i - \b e_j$.
Without loss of generality, as the hypersimplex is invariant under reordering coordinates, we suppose $c_1 < c_2 < \dots < c_n$.
Note however that there can exist $\b v,\b w\in V(\HypSimpl)$ with $\inner{\b v,\b c} = \inner{\b w, \b c}$ (when $\b v$ and $\b w$ are not adjacent vertices).
For drawings, we take $\b c = (1,2,\dots,n)$.
The vector $\b c\in\R^n$ will be fixed for the rest of this paper (apart from \Cref{ssec:NonGenericDirection}).
See \Cref{fig:MPP42} (Left) for an example.
We let \defn{$\MPPHypSimpl := \MPP[\HypSimpl][\b c]$} to ease notations.

We let $\vmin = (1,\dots,1,0,\dots,0)\in V(\HypSimpl)$ be the vertex of $\HypSimpl$ minimizing $\inner{\b v,\b c}$ for $\b v\in V(\HypSimpl)$,
and $\vmax = (0,\dots,0,1,\dots,1)\in V(\HypSimpl)$  the vertex of $\HypSimpl$ maximizing $\inner{\b v,\b c}$ for $\b v\in V(\HypSimpl)$.
We now interpret the conditions for being a coherent monotone path on a polytope for the case of the hypersimplex $\HypSimpl$.

\begin{definition}
A \defn{monotone path of vertices} $P = (\b v_1,...,\b v_r)$ on $\HypSimpl$ is an ordered list of vertices of $\HypSimpl$ with $\b v_1 = \vmin $, $\b v_r = \vmax$, and for all\footnote{We denote $[k] := \{1, 2, \dots, k\}$.} $i\in [r-1]$, $(\b v_i, \b v_{i+1})$ is an \defn{improving edge} of $\HypSimpl$ for $\b c$, \ie an edge of $\HypSimpl$ with $\inner{\b v_i, \b c} < \inner{\b v_{i+1}, \b c}$.
The \defn{length} of $P$ is $r$.

For all $i\in [r-1]$, the vertices $\b v_i$ and $\b v_{i+1}$ share an edge of $\HypSimpl$. Thus, instead of considering $P$ as a list of vertices, we emphasize what changes and what remains between $\b v_i$ and $\b v_{i+1}$ by storing the enhanced steps of $P$.
The \defn{$i$-th enhanced step} of $P$ is denoted $\enhancedStepZ$ with: 
\begin{compactenum}[$\bullet$]
\item $Z$ the common support, \ie $Z = \supp[\b v_i]\cap\supp[\b v_{i+1}]$.
\item $x$ the only index in the support of $\b v_i$ that is not in the support of $\b v_{i+1}$, \ie $\{x\} = \supp[\b v_i]\ssm\supp[\b v_{i+1}]$.
\item $y$ the only index in the support of $\b v_{i+1}$ that is not in the support of $\b v_i$, \ie $\{y\} = \supp[\b v_{i+1}]\ssm\supp[\b v_i]$.
\end{compactenum}

The list of enhanced steps of $P$ is denoted by $\c S(P)$.
The map $P \mapsto \c S(P)$ is obviously injective.
When $\enhancedStep[a][b][C]$ is the $i$-th enhanced step, and $\enhancedStepZ$ the $j$-th one, with $i < j$, we write $\enhancedStep[a][b][C] \prec \enhancedStepZ$, and say that $\enhancedStep[a][b][C]$ \defn{precedes} $\enhancedStepZ$.
Note that knowing all the steps $a\to b$, without their enhancement allows to retrieve the path and hence the enhancement.

A monotone path is called \defn{coherent} when it corresponds to a vertex of $\MPPHypSimpl$.
\end{definition}

\newpage \begin{proposition}\label{prop:CaptureCriterion}
For $\b\omega\in \R^n$, let $\pi^{\b\omega} : \R^n \to \R^2$ be the projection $\pi^{\b\omega}(\b x) = \bigl(\inner{\b x, \b c}, ~\inner{\b x, \b \omega}\bigr)$.
In particular, if $\b v\in V(\HypSimpl)$ is a vertex, then $\pi^{\b\omega}(\b v) = \bigl(\sum_{i\in \supp[\b v]} c_i,~\sum_{i\in \supp[\b v]} \omega_i\bigr)$.

A monotone path of vertices $P = (\b v_1, \dots, \b v_r)$ is coherent if and only if there exists $\b\omega\in\R^n$ such that for all $1\leq i\leq r-1$:
$$\text{for all } J\in\binom{[n]}{k},~~~~~~~~~ \sum_{j\in J} c_j > \sum_{p\in \supp[\b v_i]} c_p ~\Longrightarrow~ \slope\bigl(\supp[\b v_i],\, J\bigr) < \slope\bigl(\supp[\b v_i],\, \supp[\b v_{i+1}]\bigr)$$
where $\slope\bigl(I, J\bigr) = \frac{\sum_{j\in J} \omega_j - \sum_{i\in I} \omega_i}{\sum_{j\in J} c_j - \sum_{i\in I} c_i}$ is the slope between the point $\bigl(\sum_{i\in I} c_i,  ~\sum_{i\in I} \omega_i\bigr)$ and the point $\bigl(\sum_{i\in J} c_j,  ~\sum_{j\in J} \omega_j\bigr)$, see \Cref{fig:ProjectionExample} (Left).
We say that such $\b\omega$ \defn{captures} $P$.
\end{proposition}

\begin{proof}
By \Cref{thm:FaceLatticeMPP}, a monotone path of vertices $P$ is coherent if and only if there exists $\b\omega\in\R^n$ such that the upper path of the polygon $\pi^{\b\omega}(\HypSimpl)$ is precisely
$\pi^{\b\omega}(P) := \bigl(\pi^{\b\omega}(\b v)\bigr)_{\b v \in P}$ (remember we take the upper faces instead of the lower faces by convention in this article).

If $\b v\in V(\HypSimpl)$ is in the upper path of $\pi^{\b \omega}(\HypSimpl)$, then the next vertex in the upper path is the improving neighbor $\b v'$ of $\b v$ that maximizes the slope $\frac{\inner{\b v' - \b v,\b \omega}}{\inner{\b v' - \b v,\b c}}$.
As $\b c$ is generic for $\HypSimpl$, the vertex $\b u\in V(\HypSimpl)$ maximizing this slope is necessarily an improving neighbor of $\b v$, as the pre-image of edge $\bigl[\pi^{\b\omega}(\b v), \pi^{\b\omega}(\b u)\bigr]$ in the polygon $\pi^{\b\omega}(\HypSimpl)$ is an (improving) edge in $\HypSimpl$.
Consequently, the condition stated in the proposition is both necessary and sufficient.
\end{proof}

\begin{figure}
    \centering
    \includegraphics[width=0.98\linewidth]{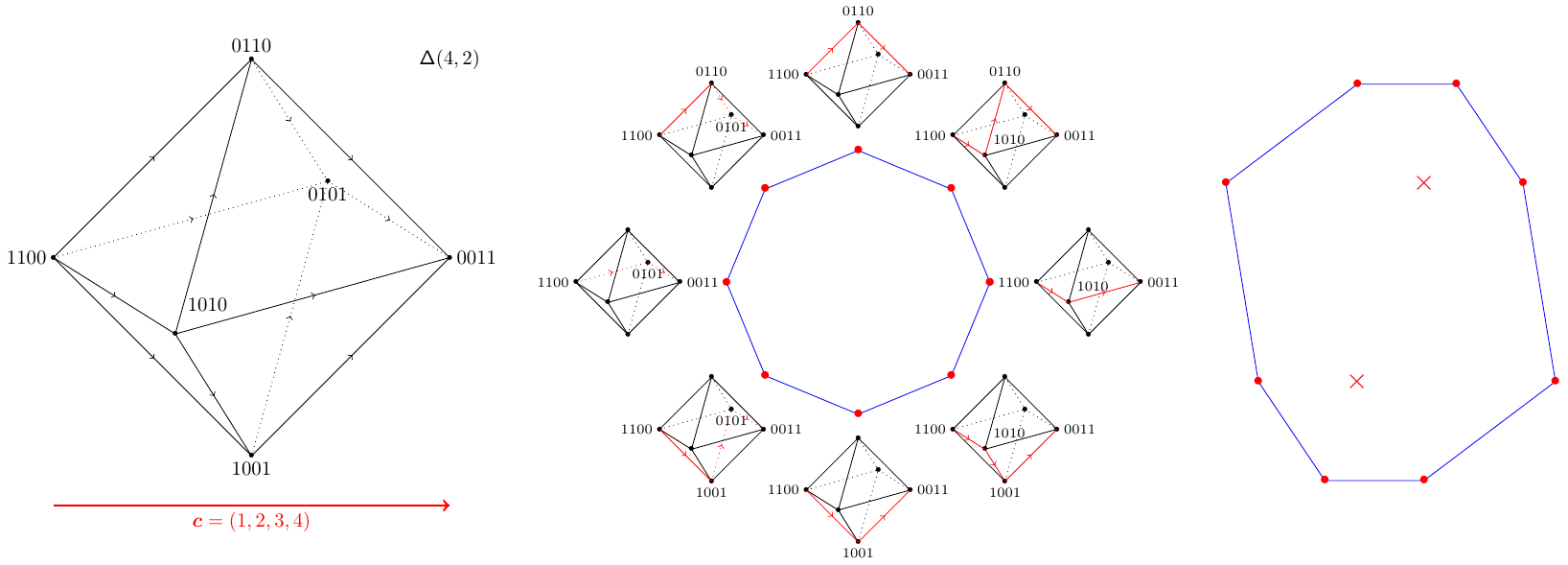}
    \caption[Monotone path polytope of the hypersimplex $\polytope{\Delta}(4,2)$]{(Left) The $(4,2)$-hypersimplex lives in the hyperplane $\{\b x ~;~ \sum_{i=1}^4 x_i = 2\}$ inside $\R^4$.
    (Middle) The monotone path polytope $\MPPHypSimpl[4][2]$ is an octagon,
    each vertex of which is labeled by the corresponding (coherent) monotone path, drawn on $\HypSimplTwo[4]$.
    (Right) Actually, $\MPPHypSimplTwo[4]$ is not a \textit{regular} octagon, but the octagon depicted here, the two crosses correspond to the two monotone paths on $\HypSimplTwo[4]$ which are not coherent.}
    \label{fig:MPP42}
\end{figure}

\begin{example}\label{exmp:Coherent32}
The hypersimplex $\HypSimplTwo[3]$ is a triangle, \ie a simplex of dimension $2$.
By Billera--Sturmfels' \Cref{thm:BS92-Simplex}, for any $\b c\in\R^3$, its monotone path polytope is a cube of dimension~$1$:
it has $2$ vertices, one corresponding to the path of length $3$, and the other one corresponding to the path of length~$2$.
\end{example}

\begin{example}\label{exmp:Coherent42}
On the hypersimplex $\HypSimplTwo[4]$, for $\b c = (1,2,3,4)$ there are 8 coherent monotone paths, and 2 non-coherent monotone paths.
The 8 coherent monotone paths correspond to the vertices of the octagon $\MPPHypSimpl[4][2]$ depicted in \Cref{fig:MPP42} (Right).
There are 4 coherent monotone paths of length 3, and 4 coherent monotone paths of length 4.
On the other side, the 2 non-coherent monotone paths are of length 5:
$(1100, 1010, \mathbf{0110}, 0101, 0011)$ and $(1100, 1010, \mathbf{1001}, 0101, 0011)$, in bold are the vertices that differ between the two paths.

With a quick jotting, one can prove that for all $\b c\in \R^4$, the same holds: for all $\b c\in\R^4$, the coherent monotone paths are exactly the same.
This can also be retrieved from \cite[Theorem~3.2]{BlackDeLoera2021monotone} as $\HypSimplTwo[4]$ is the cross-polytope of dimension 3.
\end{example}

\begin{figure}[t]
    \centering
    \includegraphics[width=0.99\linewidth]{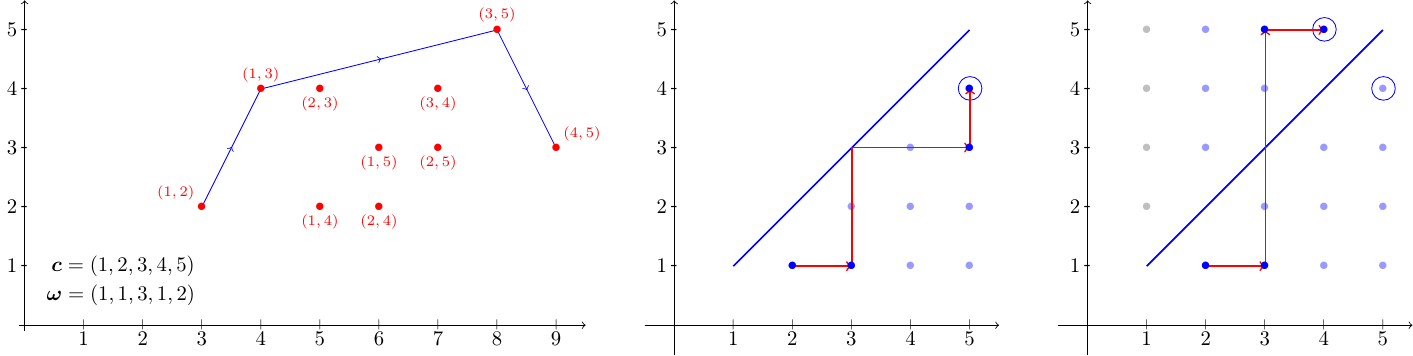}
    \caption[Example of a monotone path and its diagonal-avoiding lattice path]{(Left) For the given $\b c$ and $\b \omega$, the hypersimplex $\HypSimplTwo[5]$ is projected onto the 10 points drawn, where each vertex of $\HypSimplTwo[5]$ is indicated by its support.
    The coherent path $P$ captured is drawn in blue.
    (Middle and Right) $P$ corresponds to the diagonal-avoiding path depicted on the right, while associating $P$ to lattice points $(x,y)$ with $x < y$ give the middle figure.}
    \label{fig:ProjectionExample}
\end{figure}

\begin{example}
To be able to draw the monotone path polytope $\MPPHypSimpl$ of the hypersimplex $\HypSimpl$, we need that $\dim \HypSimpl \leq 4$, so that $\dim \MPPHypSimpl \leq 3$.
This forces $n \leq 5$.
Moreover, remember that $\HypSimpl$ is linearly isomorphic to $\HypSimpl[n][n-k]$.
For $n = 3$, \Cref{exmp:Coherent32} deals with $k = 2$ (and thus $k = 1$ by symmetry).
For $n = 4$, \Cref{exmp:Coherent42} deals with $k = 2$, while $\HypSimpl[4]$ with $k = 1$ and $k = 3$ are simplices and their monotone path polytopes are squares.
For $n = 5$, $\HypSimpl[5]$ with $k = 1$ and $k = 4$ are simplices and their monotone path polytopes are cubes, while $k = 2$ and $k = 3$ are equivalent and their monotone path polytope is depicted in \Cref{fig:M(52)}: it has 33 vertices, 52 edges, and 21 faces (5 octagons and  16 squares).
\end{example}

\begin{figure}[b]
    \centering
    \includegraphics[scale=1.33]{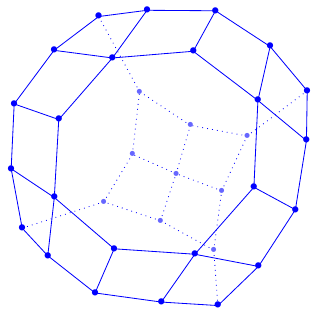}
    \caption[Monotone path polytope of the hypersimplex $\polytope{\Delta}(5,2)$]{The monotone path polytope $\MPPHypSimplTwo[5]$ of the hypersimplex $\HypSimplTwo[5]$ (not depicted with its true coordinates, but as a graph).
    This hypersimplex is linearly equivalent to $\HypSimpl[5][3]$.
    As $\HypSimplTwo[5]$ has 5 facets linearly equivalent to $\HypSimplTwo[4]$, its monotone path polytope $\MPPHypSimplTwo[5]$ has 5 facets which are isomorphic to $\MPPHypSimplTwo[4]$ \ie octagons.}
    \label{fig:M(52)}
\end{figure}

\section{A necessary criterion for coherent paths on $\HypSimpl$}\label{ssec:NecessaryCriterionHypSimpl}

\subsection{Enhanced steps criterion}\label{ssec:EnhancedStepCriterion}

For a given $\b\omega$, one can efficiently compute the monotone path that $\b\omega$ captures, using \Cref{prop:CaptureCriterion}.
However, we want to answer is the converse question: how to characterize the set of all coherent paths on the hypersimplex?
We show a necessary criterion for a monotone path to be coherent, and then, we prove that this criterion is sufficient in the case of the $(n,2)$-hypersimplices (but not in general).

\begin{theorem}\label{thm:EnhancedStepsCriterion}
If a path $P$ is coherent, then for all couples of enhanced steps $\enhancedStep[i][j][A] \prec \enhancedStepZ$ with $x < j$, one has $j \in Z$ or $x \in A$.
\end{theorem}

Before proving this criterion, we introduce a simple but powerful lemma.
It states that there exist only two kinds of triangles in the plane: upwards pointing ones $\triangle$ and downwards pointing ones $\nabla$.
This lemma is easy and well-known, but we state it here to make the section self-contained.

\begin{lemma}\label{lem:ConvexityLemma}
For three points in the plane $(x_1,y_1)$, $(x_2,y_2)$ and $(x_3,y_3)$ with $x_1 < x_2 < x_3$, 
denote the slopes by $\tau(1,2) = \frac{y_2-y_1}{x_2-x_1}$, $\tau(2,3) = \frac{y_3-y_2}{x_3-x_2}$ and $\tau(1,3) = \frac{y_3-y_1}{x_3-x_1}$.
Then $\tau(1,3)$ is a convex combination of the slopes $\tau(1,2)$ and $\tau(2,3)$.
In particular, if $\tau(1,2) < \tau(1,3)$, then $\tau(1,3) < \tau(2,3)$ (and conversely if $\tau(1,2) > \tau(1,3)$, then $\tau(1,3) > \tau(2,3)$).
\end{lemma}

\begin{proof}
One has the convex combination:
$\tau(1,3) = \frac{x_2-x_1}{x_3-x_1}\tau(1,2) + \frac{x_3-x_2}{x_3-x_1}\tau(2,3)$.
\end{proof}

\begin{proof}[Proof of \Cref{thm:EnhancedStepsCriterion}]
Suppose $\enhancedStep[i][j][A] \prec \enhancedStepZ\in \c S(P)$ with $x < j$ (so $c_x < c_j$), $j \notin Z$ and $x \notin A$.
Fix $\b \omega\in\R^n$ that captures $P$.
Then consider $\b v_1$, $\b v_2$ with $\supp[\b v_1] = A\cup\{i\}$, $\supp[\b v_2] = A\cup\{j\}$, and $\b v_3$, $\b v_4$ with $\supp[\b v_3] = Z\cup\{x\}$, $\supp[\b v_4] = Z\cup\{y\}$, see \Cref{fig:CoherenceCriterionProof}.
These are 4 vertices of $\HypSimpl$ in the path $P$.
Abusing notation, we write $\slope(\b u,\b v)$ instead of $\slope\bigl(\supp[\b u],\supp\bigr)$ in what follows.

As $x \notin A$, there exists $\b u_1\in V(\HypSimpl)$ with $\supp[\b u_1] = A\cup\{x\}$, thus $\b v_2$ is an improving neighbor of $\b u_1$.
As $j \notin Z$, there exists $\b u_2\in V(\HypSimpl)$ with $\supp[\b u_2] = Z\cup\{j\}$, thus $\b u_2$ is an improving neighbor of $\b v_3$.
First observe that, $\slope\bigl(\b v_1, \b u_1\bigr) < \slope\bigl(\b v_1, \b v_2\bigr)$ by \Cref{prop:CaptureCriterion},
thus $\slope\bigl(\b u_1, \b v_2\bigr) > \slope\bigl(\b v_1,\b v_2\bigr)$ by \Cref{lem:ConvexityLemma} applied in the triangle $\pi_{\b \omega}(\b v_1)$, $\pi_{\b \omega}(\b v_2)$, $\pi_{\b \omega}(\b u_1)$.
Moreover, $\slope\bigl(\b v_3, \b u_2\bigr) < \slope\bigl(\b v_3, \b v_4\bigr)$ by \Cref{prop:CaptureCriterion}.
As $\pi_{\b\omega}(P)$ is convex: $\slope\bigl(\b v_1, \b v_2\bigr) > \slope\bigl(\b v_3, \b v_4\bigr)$ because the second step comes later in the path.
But then: $\slope\bigl(\b u_1,\b v_2\bigr) < \slope\bigl(\b v_3,\b u_2\bigr)$, while in the meantime:
$\slope\bigl(\b u_1,\b v_2\bigr) = \frac{\omega_j-\omega_x}{c_j - c_x} = \slope\bigl(\b v_3,\b u_2\bigr)$.
This contradiction proves the theorem.
\end{proof}

\begin{figure}[b]
    \centering
    \begin{center}
\begin{tikzpicture}[scale=1.5,xscale=1.5,yscale=1.25]
\newdimen\RR
\RR = 1cm

\draw[->] (-2,-0.25) -- (2,-0.25);
\draw[->] (-1.75,-0.75) -- (-1.75,2);

\draw[blue] (0:\RR) \foreach \i in {1,...,6}{-- ({180*\i/6}:\RR)};
\draw[very thick, blue] (180/6:\RR) -- (180/3:\RR);
\draw[very thick, blue] (180*2/3:\RR) -- (180*5/6:\RR);

\draw[blue] (1.3,0.1) node{$\pi_{\b \omega}(P)$};

\draw[red] (180*5/6:1.05*\RR) node[anchor=east]{\begin{small}$\pi_{\b \omega}(\b v_1)$\end{small}};

\draw[red] (180*4/6:1*\RR) node[anchor=south east]{\begin{small}$\pi_{\b \omega}(\b v_2)$\end{small}};

\draw[red] (180*2/6:1*\RR) node[anchor=south west]{\begin{small}$\pi_{\b \omega}(\b v_3)$\end{small}};

\draw[red] (180*1/6:1.05*\RR) node[anchor=west]{\begin{small}$\pi_{\b \omega}(\b v_4)$\end{small}};

\draw[red] (180*5/6:0.8*\RR) node{\begin{tiny}$\bullet$\end{tiny}};
\draw[blue, dotted] (180*5/6:0.8*\RR) -- (180*4/6:\RR);
\draw[red] (180*5/6:0.82*\RR) node[anchor=north west]{\begin{small}$\pi_{\b \omega}(\b u_1)$\end{small}};

\draw[red] (180*1/6:0.8*\RR) node{\begin{tiny}$\bullet$\end{tiny}};
\draw[blue, dotted] (180*1/6:0.8*\RR) -- (180*2/6:\RR);
\draw[red] (180*1/6:0.82*\RR) node[anchor=north east]{\begin{small}$\pi_{\b \omega}(\b u_2)$\end{small}};

\foreach \i in {0,...,6}{\draw[red] ({180*\i/6}:\RR) node{\begin{tiny}$\bullet$\end{tiny}};}

\draw (-1.7,1.9) node[anchor=west]{\begin{small}$\supp[\b v_1] = A\cup\{i\}$\end{small}};
\draw (-1.7,1.65) node[anchor=west]{\begin{small}$\supp[\b v_2] = A\cup\{j\}$\end{small}};
\draw (0.9,1.9) node[anchor=west]{\begin{small}$\supp[\b v_3] = Z\cup\{x\}$\end{small}};
\draw (0.9,1.65) node[anchor=west]{\begin{small}$\supp[\b v_4] = Z\cup\{y\}$\end{small}};
\draw (-1.7,1.4) node[anchor=west]{\begin{small}$\supp[\b u_1] = A\cup\{x\}$\end{small}};
\draw (0.9,1.4) node[anchor=west]{\begin{small}$\supp[\b u_2] = Z\cup\{j\}$\end{small}};
\end{tikzpicture}
\end{center}
    \caption[Illustration for the proof of \Cref{thm:EnhancedStepsCriterion}]{Illustration for the proof of \Cref{thm:EnhancedStepsCriterion}: Dotted slopes shall be both equal to $\frac{\omega_j-\omega_x}{c_j-c_x}$. It is impossible if $\pi_{\b \omega}(\b u_1)$ and $\pi_{\b \omega}(\b u_2)$ are below $\pi_{\b \omega}(P)$: either $\b u_1$ or $\b u_2$ is not a vertex of $\HypSimpl$.}
    \label{fig:CoherenceCriterionProof}
\end{figure}

\subsection{Comparison with coherent paths for a non-generic direction}\label{ssec:NonGenericDirection}

Before exploring the sufficiency of the enhanced steps criterion, we want to connect our work to the pre-existing result on the coherent monotone paths on $\HypSimpl$ for non-generic directions.
In \cite[Section 8]{BlackSanyal-FlagPolymatroids}, the authors study coherent monotone paths on matroid base polytopes $\polytope{B}_M$ for a matroid $M$ on $n$ elements, for the non-generic direction $\b e_S := \sum_{i\in S} \b e_i$ for $S\subseteq [n]$.
They show:

\begin{theorem}\emph{(\cite[Theorem 8.1]{BlackSanyal-FlagPolymatroids}.)}
For $S\subseteq[n]$, all $\b e_S$-monotone paths on $\polytope{B}_M$ are coherent.
\end{theorem}

If $M = U(n, k)$ is the uniform matroid of rank $k$ on $n$ elements, then $\polytope{B}_M = \HypSimpl$.

Up to symmetry, it is enough to consider $S = [n-s+1, n]$.
Clearly, $\b e_S$ is not a generic direction for $\HypSimpl$:
one would like that a coherent path $P$ of $(\HypSimpl, \b e_S)$ is a sub-path of a coherent path $P'$ of $(\HypSimpl, \b c)$ for a generic~$\b c$.
We show it holds.
An $\b e_S$-monotone path on $\HypSimpl$ is a sequence of bases $\c B_1, \dots, \c B_{k+1} \in \binom{[n]}{k}$, where $\c B_{i+1} = \bigl(\c B_i \ssm\{a_i\}\bigr) \cup\{b_i\}$ with $a_i\notin S$, $b_i\in S$, and $a_1, \dots, a_k$, $b_1, \dots, b_k$ distinct, where $\c B_1$ minimizes $|\c B\cap S|$ among all $\c B \in \binom{[n]}{k}$.
For us, this means that the enhanced steps of the path $P$ are $\enhancedStep[a_i][b_i][\c B_i \ssm\{a_i\}]$.
This respects the criterion of \Cref{thm:EnhancedStepsCriterion}: if $\enhancedStep[a_i][b_i][\c B_i \ssm\{a_i\}] \prec \enhancedStep[a_j][b_j][\c B_j \ssm\{a_j\}]$ then $i < j$, ensuring $b_i \in \c B_j \ssm\{a_j\}$ and $a_j \in \c B_i \ssm\{a_i\}$.

Hence, if $P$ is a coherent path for $\b e_S$, then it respects the enhanced steps criterion for a generic $\b c$.
However, this does not yet prove that there exists a path $P'$, coherent for some generic $\b c$, such that $P$ is a sub-path of $P'$, because $P$ does not necessarily start at the minimal vertex of $\HypSimpl$ (nor end at a maximal vertex), but at any vertex minimizing $\b e_S$.
To construct such $P'$ we define the part that goes from the minimal vertex $\b v_{\min} = \sum_{q\in [k]} \b e_q$ to the starting vertex $\sum_{q\in \c B_1} \b e_q$, and symmetrically for the end of the path.
In order to satisfy the enhanced steps criterion, these new steps are added in a certain order.
The technical definition thereafter guaranties such an ordering:

\begin{definition}
For $S = [n-s+1, n]$, and a (coherent) $\b e_S$-monotone path $P$, with $\c B_1 \subseteq \binom{[n]\ssm S}{k}$, and enhanced steps $\enhancedStep[a_i][b_i][\c B_i \ssm\{a_i\}]$ (ordered by $i$), we construct the \defn{lifted path $P^\uparrow$} defined by its steps as follows.
Let $i_1 > \dots > i_m$ be such that $\c B_1 \ssm[k] = \{a_{i_1}, \dots, a_{i_m}\}$.
Let $[k]\ssm \c B_1 = \{x_1, \dots, x_m\}$, and $j_1 > \dots > j_r$ be such that $\c B_{k+1} \ssm[n-k+1, n] = \{b_{j_1}, \dots, b_{j_r}\}$ and $[n-k+1, n] = \{y_1, \dots, y_r\}$.
$$\c S(P^\uparrow) = \left\{\begin{array}{l}
x_1 \to a_{i_1} \prec x_2 \to a_{i_2} \prec \dots \prec x_m \to a_{i_m} \prec \\
a_1 \to b_1 \prec a_2 \to b_2 \prec \dots \prec a_k \to b_k \prec \\
b_{j_1} \to y_1 \prec b_{j_2} \to y_2 \prec \dots \prec b_{j_r} \to y_r
\end{array}\right.$$
\end{definition}

\begin{proposition}\label{prop:LiftedPath}
For all $s\in [n-1]$, and $P$ an $\b e_{[n-s+1, n]}$-monotone path, the lifted path $P^\uparrow$ satisfies the enhanced steps criterion, \ie for all $\enhancedStep[i][j][A] \prec \enhancedStepZ\in \c S(P^\uparrow)$, if~$x < j$, then $x\in A$ or $j\in Z$.
\end{proposition}

\begin{proof}
Steps in $\c S(P^\uparrow)$ are either of form $x_p \to a_{i_p}$, or $a_p\to b_p$, or $b_{j_p} \to y_p$.
The enhanced steps criterion holds between two steps of the second type: if $\enhancedStep[a_i][b_i][\c B_i \ssm\{a_i\}] \prec \enhancedStep[a_j][b_j][\c B_j \ssm\{a_j\}]$, then $i < j$, ensuring $b_i \in \c B_j \ssm\{a_j\}$ and $a_j \in \c B_i \ssm\{a_i\}$.
Similarly, it holds between any two steps of the same type.
Now pick $\enhancedStep[x_p][a_{i_p}][A] \prec \enhancedStep[a_q][b_q][Z]$.
Suppose $a_q < a_{i_p}$.
If $q < i_p$, then $i_p \in \c B_q$, 
so $a_{i_p}\in Z$.
If $i_p < q$, then either $a_q\in [k]\cap\c B_1$, or there is a step $x\to a_q$ preceding the step $\enhancedStep[x_p][a_{i_p}][A]$.
In both situations, $a_q\in A$. 
Consequently, the enhanced steps criterion holds.
The case of a step of second kind and of third kind is similar.
Finally, if $\enhancedStep[x_p][a_{i_p}][A] \prec \enhancedStep[b_{j_q}][y_q][Z]$, then $a_{i_p} < b_{j_q}$ because $S = [n-s+1, n]$, $a_{i_p}\notin S$ and $b_{j_q}\in S$, so the enhanced steps criterion holds.
\end{proof}

We have shown that, for $\HypSimpl$, all coherent paths for a non-generic direction $\b e_S$ for $S\subseteq[n]$ can be lifted to a monotone path for a generic direction satisfying the enhanced steps criterion: a path that ``should'' be coherent.
For general $k$, we are not able to prove that this path is indeed coherent, because the enhanced steps criterion is not sufficient.
However, for $k = 2$ this criterion is sufficient: we conclude here, even though this sufficiency is proven in the next section.

\begin{corollary}
For all $n\geq 1$ and $S\subseteq[n]$, for all coherent $\b e_S$-monotone path $P$ on $\HypSimplTwo$, there exists a coherent $\b c$-monotone path $P'$ on $\HypSimplTwo$ for a generic $\b c$ such that $P$ is a sub-path of $P'$.
\end{corollary}

\begin{proof}
According to \Cref{prop:LiftedPath}, the path $P^\uparrow$ respects the hypothesis of \Cref{cor:EnhancedStepCriterionIsCoherence}.
\end{proof}

\section{Sufficiency of the enhanced steps criterion in the case $\HypSimplTwo$}\label{ssec:SufficiencyCriterionHypSimplTwo}

\subsection{From coherent paths to lattice paths}\label{ssec:TowardsLatticePaths}

We are going to prove that for $\HypSimplTwo$, the criterion of \Cref{thm:EnhancedStepsCriterion} is actually sufficient.
To this end, we want to associate monotone paths on $\HypSimpl$ with some lattice paths on the integer grid~$[n]^k$.
A first idea to do so would be to associate to each vertex $\b v_i$ in the path $P = (\b v_1, \dots, \b v_r)$ a point $\b\ell_i=(\ell_{i,1},\dots,\ell_{i,k})\in[n]^k$ satisfying $\{\ell_{i,1},\dots,\ell_{i,k}\} = \supp[\b v_i]$.
This leaves $k!$ possible choices for $\b\ell_i$.
Even though a \emph{natural} choice would be to impose $\ell_{i,1}<\dots<\ell_{i,k}$, we will prefer another one.
Indeed, as $\b v_i$ and $\b v_{i+1}$ form an edge of $\HypSimpl$, there is only one index differing between $\supp[\b v_i]$ and $\supp[\b v_{i+1}]$, so we will impose that $\b\ell_i$ and $\b\ell_{i+1}$ differ at only one coordinate.

Although this idea allows us to embed our problem into the realm of lattice paths, it has the drawback that $k!$ different lattice points are associated to a same vertex, see \Cref{fig:ProjectionExample} (Right).

\begin{definition}
A \defn{diagonal-avoiding lattice path} $\c L = (\b \ell_1,\dots,\b \ell_r)$ of \defn{size} $n$ and \defn{dimension} $k$ is an ordered list of points $\b\ell_i\in [n]^k$ such that:
\begin{itemize}
\item $\b\ell_1 = (k,k-1,\dots,1)$;
\item $\b\ell_r = (\ell_{r,1},\dots,\ell_{r,k})$ with $\{\ell_{r,1},\dots,\ell_{r,k}\} = \{n-k+1,\dots,n\}$;
\item for all $i\in [r]$, $\ell_{i,p} \ne \ell_{i,q}$ for all $p,q\in [k]$ with $p\ne q$;
\item for all $i\in[r-1]$, there exists $p\in [k]$ such that $\ell_{i,p} < \ell_{i+1,p}$, and $\ell_{i,q} = \ell_{i+1,q}$ for all $q\ne p$. The \defn{$i$-th enhanced step} of $\c L$ is denoted by $\enhancedStep[\ell_{i,p}][\ell_{i+1,p}][Z]$ with $Z = \{\ell_{i,q} ~;~ q\ne p\}$.
\end{itemize}

The ordered list of enhanced steps of $\c L$ is denoted by $\c S(\c L)$.
The \defn{length} of $\c L$ is $r$.
\end{definition}

To a path $P = (\b v_1, \dots, \b v_r)$ on $\HypSimpl$, one can associate a diagonal-avoiding lattice path $\scr L(P) = \bigl(\b \ell_1,\dots,\b \ell_r\bigr)$ of size $n$ and dimension $k$ defined by $\c S(\scr L(P)) = \c S(P)$, see \Cref{fig:PathExample}.

\begin{figure}[b]
    \centering
    \includegraphics[scale=0.4]{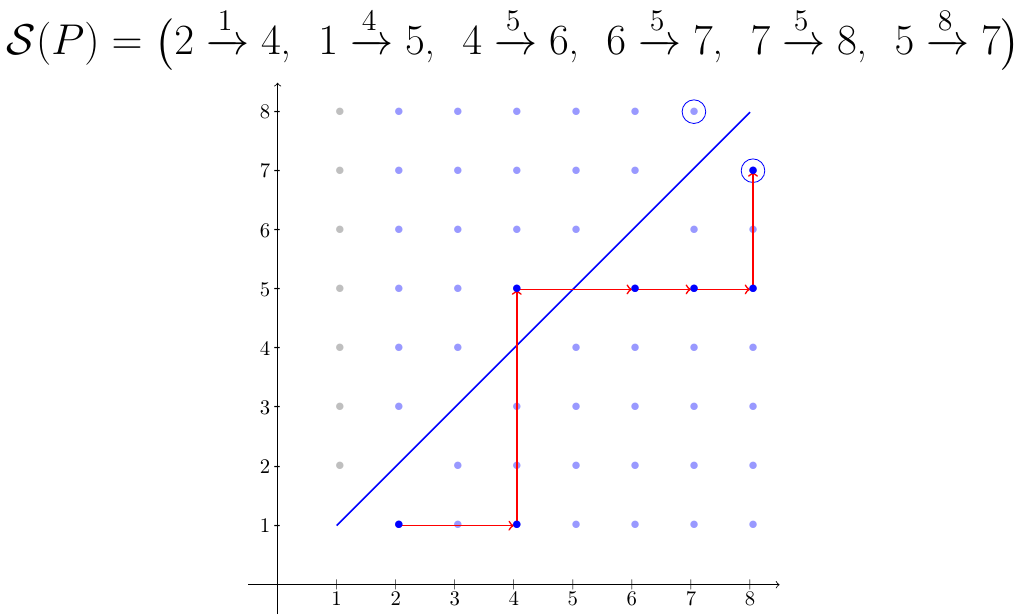}
    \caption[Lattice path associated to a given ordered list of enhanced steps]{A lattice path, and its ordered list of enhanced steps (size 8, dimension 2, length 6).}
    \label{fig:PathExample}
\end{figure}

\begin{proposition}
The map $P \mapsto \scr L(P)$ is a bijection from monotone paths on $\HypSimpl$ to diagonal-avoiding lattice paths of size $n$ and dimension $k$.
\end{proposition}

\begin{proof}
Fix a monotone path $P$ on $\HypSimpl$.
Starting at $\b \ell_1 = (k,k-1,\dots,1)$, the lattice path $\scr L(P) = (\b\ell_1,\dots,\b\ell_i,\dots,\b\ell_r)$ can be defined by induction on $i$.
Indeed, denote by $\scr L(P)_{\leq i} = (\b\ell_1,\dots,\b\ell_i)$, and suppose that for a fixed $i$: $\{\ell_{j,1},\dots,\ell_{j,k}\} = \supp[\b v_j]$ for all $j\leq i$, and the enhanced steps of $\scr L(P)_{\leq i}$ are the $(i-1)$ first enhanced steps of $P$.
Consider the $i$-th enhanced step of $P$, say $\enhancedStepZ$.
As $\{\ell_{i,1},\dots,\ell_{i,k}\} = \supp[\b v_j]$, there exists $p\in[k]$ such that $\ell_{i,p} = x$, and $\{\ell_{i,1},\dots,\ell_{i,k}\}\ssm\{\ell_{i,p}\} = Z$.
By setting $\b\ell_{i+1}$ with $\ell_{i+1,q} = \ell_{i,q}$ for $q\ne p$, and $\ell_{i+1,p} = y$, we construct $\scr L(P)_{\leq i+1}$ that fulfills the induction hypothesis.
Hence, we can define $\scr L(P)$ such that $\c S(\scr L(P)) = \c S(P)$.
By induction, $\scr L(P)$ satisfies that $\{\ell_{i,1},\dots,\ell_{i,k}\} = \supp[\b v_i]$.
Moreover, as $|\supp[\b v_i]| = k$, we know that $\ell_{i,p} \ne \ell_{i,q}$ for all $p\ne q$.
Consequently, as $\supp[\b v_i]$ and $\supp[\b v_{i+1}]$ differ by only one element, $\scr L(P)$ is a diagonal-avoiding path.

As $P\mapsto\c S(P)$ is injective, it is immediate that $P\mapsto\scr L(P)$ is also injective.

Finally, for all diagonal-avoiding paths $\c L = (\b\ell_1,\dots,\b\ell_r)$, one can construct by induction an ordered list of vertices $P_{\c L} = (\b v_1,\dots,\b v_r)$ by taking $\b v_i = \sum_{j\in\ell_i} \b e_j$.
Such a path $P_{\c L}$ is a monotone path on $\HypSimpl$ thanks to the properties of diagonal-avoiding paths.
Moreover, as $\c S(P_{\c L}) = \c S(\c L)$, the map $\c L \mapsto P_{\c L}$ is the reciprocal of $P\mapsto \scr L(P)$.
\end{proof}

\begin{remark}
It is straightforward to see that the length of $P$, \ie the number of vertices contained in $P$, equals the length of $\scr L(P)$, \ie the number of lattice points contained in $\scr L(P)$.
\end{remark}

\begin{example}\label{exmp:SmallDApaths}
For size $n = 3$,
there are 2 diagonal-avoiding paths, one of length 1 and one of length 2.
As seen in \Cref{exmp:Coherent32}, all of them are images (by $\scr L$) of coherent paths on the simplex $\HypSimplTwo[3]$.

For size $n = 4$,
there are 10 diagonal-avoiding paths, see \Cref{fig:SmallDApaths}.
As seen in \Cref{exmp:Coherent42}, 8 of them are images (by $\scr L$) of coherent paths on $\HypSimplTwo[4]$, while 2 (of length 5) come from monotone but not coherent paths on $\HypSimplTwo[4]$.

\begin{figure}
    \centering
    \includegraphics[width=0.9\linewidth]{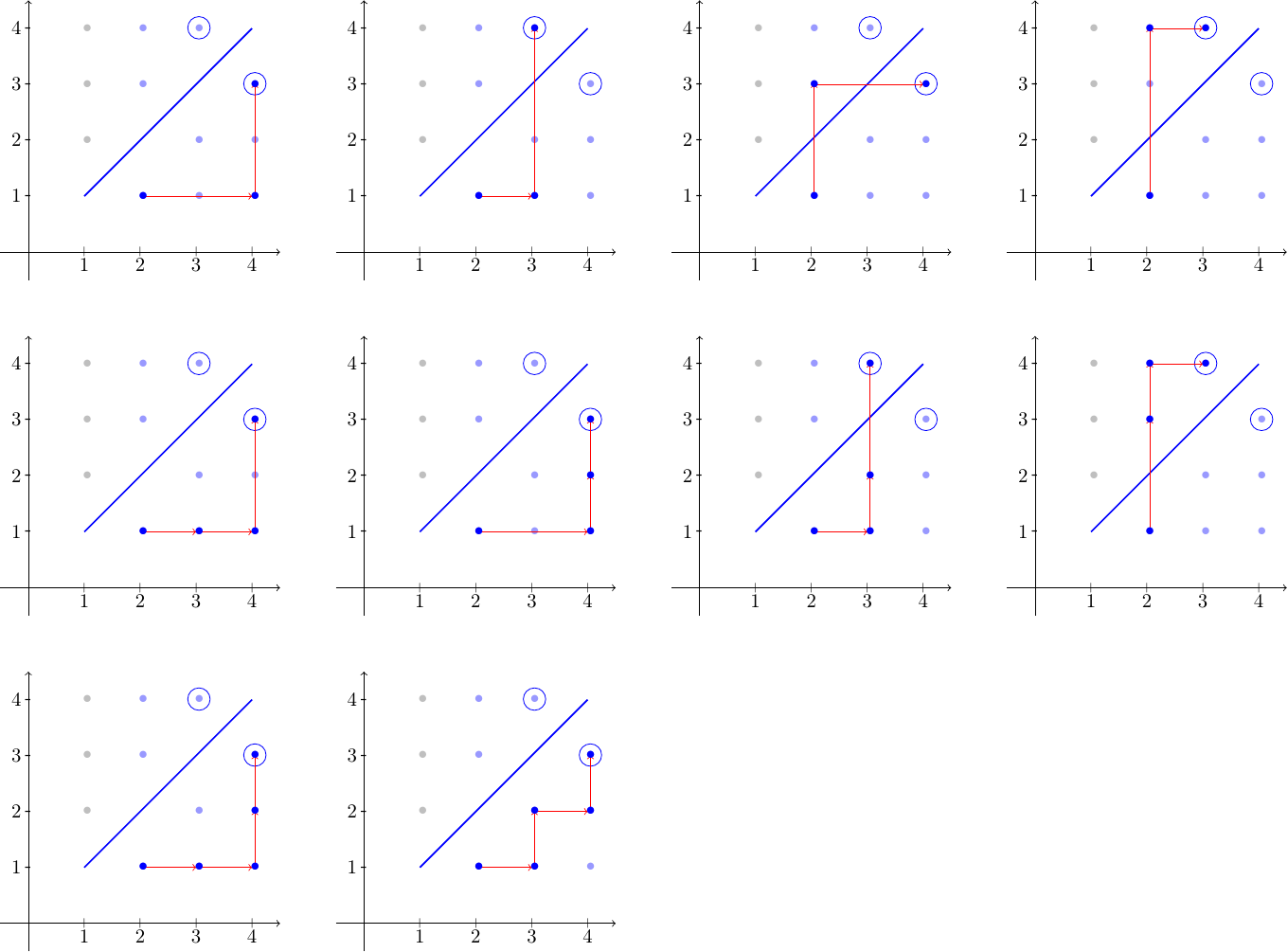}
    \caption{All $10$ diagonal-avoiding lattice paths of size 4 and dimension 2, sorted by size.}
    \label{fig:SmallDApaths}
\end{figure}
\end{example}

To ease notation, for an enhanced step of a path $P$ on $\HypSimplTwo$ or enhanced steps of diagonal-avoiding lattice paths of dimension $2$, we will write $\enhancedStep$ instead of $\enhancedStep[i][j][\{a\}]$.
We will now study diagonal-avoiding paths of dimension 2.
In particular, we will show that coherent monotone paths on $\HypSimplTwo$ are associated with a certain family of diagonal-avoiding lattice paths, and that this family respects an induction process (which is cumbersome but powerful).
To describe this induction process for our family, we need the notion of \emph{restriction} of diagonal-avoiding lattice paths, which consists of shrinking the lattice grid $[n]^2$: given a diagonal-avoiding lattice path on $[n+1]^2$, erase the points of $[n+1]^2\ssm[n]^2$ ; the path obtained on $[n]^2$ will probably not end at the right spot, but can be completed it to mimic the path you started with.
The following definition formalizes this idea.

\begin{definition}
The \defn{restriction} of a diagonal-avoiding lattice path $\c L = (\b\ell_1,\dots,\b\ell_r)$ of size $n+1$ and dimension $2$ is the diagonal-avoiding lattice path $\c L' = (\b\ell'_1,\dots,\b\ell'_s)$ of size $n$ and dimension $2$ defined by:
\begin{enumerate}
\item First, for all $i\in[r]$ define $\ell'_{i,p} = \left\{\begin{array}{ll}
    \ell_{i,p} &  \text{ if } \ell_{i,p} \ne n+1 \\
    n & \text{ else}
\end{array}\right.$ (for $p\in \{1,2\}$) with $s = r$,
\item Next, as $\b\ell'_r = (n,n)$: if $\b\ell'_{r-1} = (x,n)$ then set $\b\ell'_r = (n-1,n)$, whereas if $\b\ell'_{r-1} = (n,x)$ then set $\b\ell'_r = (n,n-1)$;
\item Finally, if $\b\ell'_i = \b\ell'_{i+1}$, then discard $\b\ell'_{i+1}$ (and keep discarding until no doubles remain).
\end{enumerate}
\end{definition}

Even though this definition seems convoluted, it has a very straightforward illustration, see \Cref{fig:Restriction}: as explained before,
draw the path $\c L$ on the $(n+1)\times (n+1)$ grid, then $\c L'$ is obtained by first restricting $\c L$ to the $n\times n$ grid, then mimicking the steps $\enhancedStep[i][j][n+1]$ of $\c L$ by introducing the steps $\enhancedStep[i][j][n]$ in $\c L'$ (and slightly modifying $\c L'$ to make it diagonal-avoiding).

\begin{figure}
    \centering
    \includegraphics[width=0.85\linewidth]{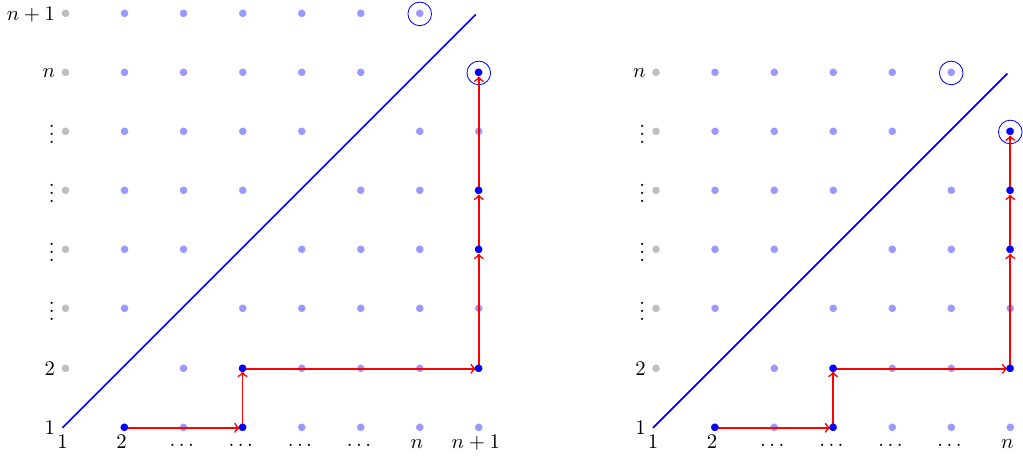}
        \caption[Restriction of a diagonal-avoiding lattice path]{The restriction of $\c L$ to $\c L'$. The enhanced steps are $\c S(\c L) = \bigl(\enhancedStep[2][4][1],~\enhancedStep[1][2][4],~\enhancedStep[4][n+1][2],~\enhancedStep[2][4][n+1],~\enhancedStep[4][5][n+1],~\enhancedStep[5][n][n+1]\bigr)$ to $\c S(\c L') = \bigl(\enhancedStep[2][4][1],~\enhancedStep[1][2][4],~\enhancedStep[4][n][2],~\enhancedStep[2][4][n],~\enhancedStep[4][5][n],~\enhancedStep[5][n-1][n]\bigr)$.}
    \label{fig:Restriction}
\end{figure}


We can now introduce the main object of this section:

\begin{definition}
A diagonal-avoiding lattice path $\c L$ of dimension 2 is said to be a \defn{coherent lattice path} if for all couples of enhanced steps $\enhancedStep \prec \enhancedStepx$ with $x < j$, we have $j = z$ or $x = a$.
\end{definition}

Now, we will study the set of coherent lattice paths of size $n$.
First, we will prove that such lattice paths can be constructed inductively.
Then, we will show that the bijection $P \mapsto \scr L(P)$ (between monotone paths and diagonal-avoiding paths)
sends coherent paths on $\HypSimplTwo$ to coherent lattice paths.
Finally, our inductive construction will allow us to count the number of coherent paths on $\HypSimplTwo$.

\subsection{The induction process}\label{ssec:InductionProcess}

\begin{theorem}\label{thm:InductionProcess}
For $n\geq 3$, let $\c L$ be a coherent lattice path of size $n+1$ and $\c L'$ its restriction of size $n$. Then $\c L'$ is coherent and $\c L$ can be reconstructed from $\c L'$ as it belongs to one of these (mutually exclusive) 12 cases:
\begin{enumerate}[(i)]
\item if $\c L'$ ends by a step $\enhancedStep[x][n][n-1]$ with $x < n-1$, then let $\c S' = \c S(\c L') \ssm \{\enhancedStep[x][n][n-1]\}$. One of the following holds (see \Cref{fig:InductionTrois}):
    \begin{enumerate}
        \item $\c S(\c L) = \c S(\c L') \cup \{\enhancedStep[n-1][n+1][n]\}$
        \item $\c S(\c L) = \c S(\c L') \cup \{\enhancedStep[n][n+1][n-1], ~\enhancedStep[n-1][n][n+1]\}$
        \item $\c S(\c L) = \c S' \cup \{\enhancedStep[x][n+1][n-1], ~\enhancedStep[n-1][n][n+1]\}$
    \end{enumerate}
\item if $\c L'$ ends by steps $\enhancedStep[x][n][y_1]$, $\enhancedStep[y_1][y_2][n], \dots, \enhancedStep[y_{m-1}][y_m][n]$ with $x < n-1$, $m \geq 3$ and $y_1 < \dots < y_m = n-1$, then let $\c S' = \c S(\c L') \ssm \{\enhancedStep[x][n][y_1], ~\enhancedStep[y_1][y_2][n], ~\dots, ~\enhancedStep[y_{m-1}][y_m][n]\}$. One of the following holds (see \Cref{fig:InductionQuatre}):
    \begin{enumerate}
        \item $\c S(\c L) = \c S(\c L') \cup \{\enhancedStep[n-1][n+1][n]\}$
        \item $\c S(\c L) = \bigl(\c S(\c L') \ssm \{\enhancedStep[y_{m-1}][n-1][n]\}\bigr) \cup \{\enhancedStep[y_{m-1}][n+1][n]\}$
        \item $\c S(\c L) = \c S' \cup \{\enhancedStep[x][n+1][y_1], ~\enhancedStep[y_1][y_2][n+1], ~\dots, ~\enhancedStep[y_{m-1}][n-1][n+1], ~\enhancedStep[n-1][n][n+1]\}$
        \item $\c S(\c L) = \c S' \cup \{\enhancedStep[x][n+1][y_1], ~\enhancedStep[y_1][y_2][n+1], ~\dots, ~\enhancedStep[y_{m-1}][n][n+1]\}$
    \end{enumerate}
\item if $\c L'$ ends by steps $\enhancedStep[x][n][y]$, $\enhancedStep[y][n-1][n]$ with $x < n$ and $y < n-1$, then let $\c S' = \c S(\c L') \ssm \{\enhancedStep[x][n][y], ~\enhancedStep[y][n-1][n]\}$. One of the following holds (see \Cref{fig:InductionCinq}):
    \begin{enumerate}
        \item $\c S(\c L) = \c S(\c L') \cup \{\enhancedStep[n-1][n+1][n]\}$
        \item $\c S(\c L) = \bigl(\c S(\c L') \ssm \{\enhancedStep[y][n-1][n]\} \bigr) \cup \{\enhancedStep[y][n+1][n]\}$
        \item $\c S(\c L) = \bigl(\c S(\c L') \ssm \{\enhancedStep[y][n-1][n]\} \bigr) \cup \{\enhancedStep[n][n+1][y], ~\enhancedStep[y][n][n+1]\}$
        \item $\c S(\c L) = \c S' \cup \{\enhancedStep[x][n+1][y], ~\enhancedStep[y][n-1][n+1], ~\enhancedStep[n-1][n][n+1]\}$
        \item $\c S(\c L) = \c S' \cup \{\enhancedStep[x][n+1][y], ~\enhancedStep[y][n][n+1]\}$
    \end{enumerate}
\end{enumerate}
\end{theorem}

\begin{figure}[h]
    \centering
    \includegraphics[scale=0.85]{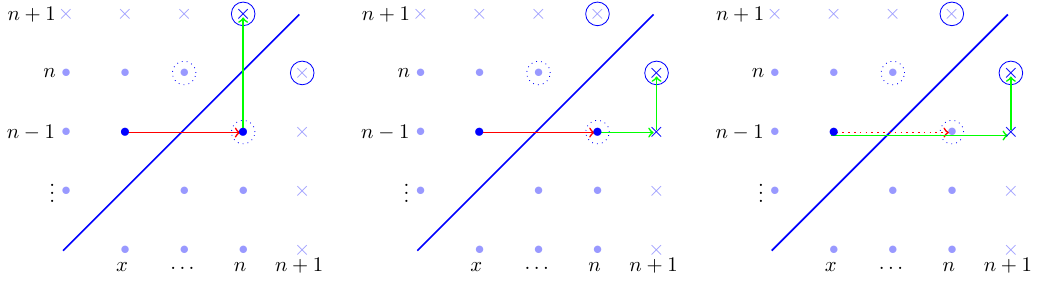}
    \caption{All $3$ paths of size $n+1$ that restrict to a path of size $n$ of type $(i)$ in \Cref{thm:InductionProcess}.}
    \label{fig:InductionTrois}

    \centering
    \includegraphics[scale=0.63]{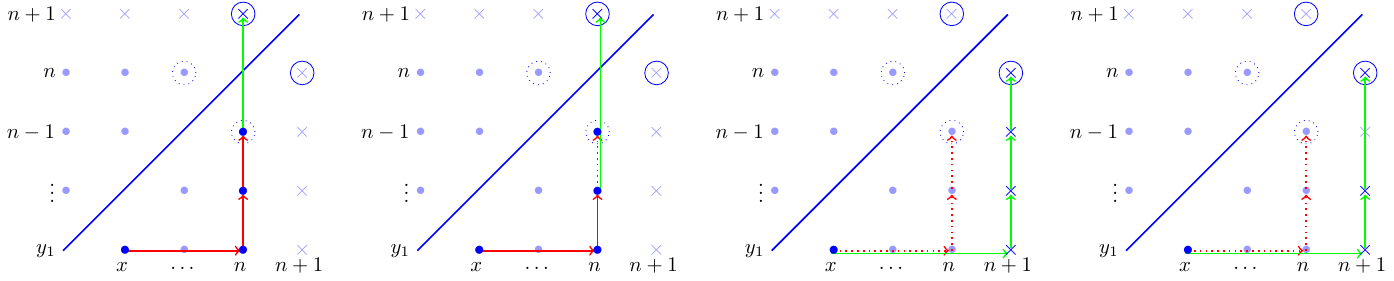}
    \caption{All $4$ paths of size $n+1$ that restrict to a path of size $n$ of type $(ii)$ in \Cref{thm:InductionProcess}.}
    \label{fig:InductionQuatre}

    \centering
    \includegraphics[scale=0.5]{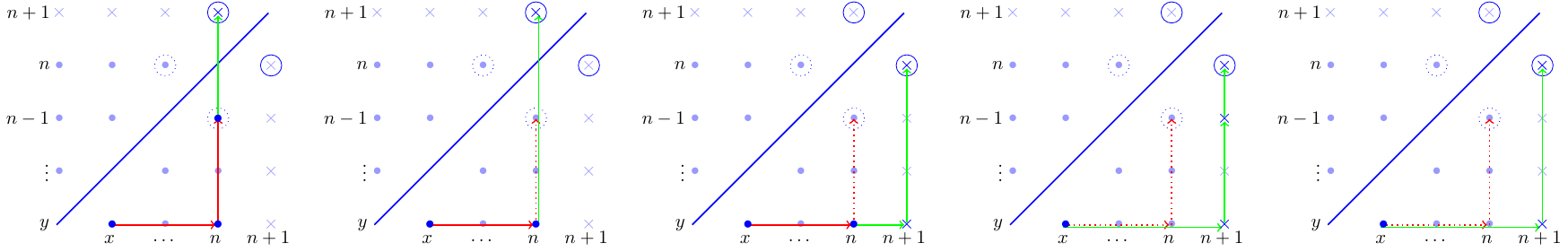}
    \caption{All $5$ paths of size $n+1$ that restrict to a path of size $n$ of type $(iii)$ in \Cref{thm:InductionProcess}.}
    \label{fig:InductionCinq}
\end{figure}

\begin{proof}
Observe first that if $\c L$ is a coherent lattice path of size $n+1$, then its restriction $\c L'$ of size $n$ is also coherent. 
Indeed, if $\enhancedStep \prec\enhancedStepx\in\c S(\c L')$ with $x < j$, then the following hold:
\begin{compactenum}[$\bullet$]
\item either $\enhancedStep\in\c S(\c L)$, or $\enhancedStep[i][n+1]\in\c S(\c L)$ and $j=n$;
\item either $\enhancedStepx\in\c S(\c L)$, or $\enhancedStepx[x][n+1]\in\c S(\c L)$ and $y = n$, or $\enhancedStep[x][y'][n+1]\in\c S(\c L)$ and $z = n$ and $y'\in\{y,n\}$.
\end{compactenum}
As $\c L$ is coherent, $x = a$ or $j = z$ in all cases except if $\enhancedStep \in \c S(\c L)$ with $j=n$ and $\enhancedStepx[x][y][n+1]\in\c S(\c L)$ with $x > a$. But then $x > a > j$ and $x \ne a$, $j \ne n+1$, breaking the coherence of $\c L$.
Now, we prove that all 12 cases lead to coherent paths, and that there is no other coherent path of size $n+1$.

We say that $\enhancedStep \prec\enhancedStepx$ are \emph{mutually coherent} if $x \geq j$, or if $x < j$ and $j = z$ or $x = a$.

\underline{Case $(i)(a)$, $(ii)(a)$ and $(iii)(a)$} If $\enhancedStep\in \c S(\c L')$ satisfies $n-1 < j$, then $j = n$ so adding $\enhancedStep[n-1][n+1][n]$ to $\c S(\c L')$ does not infringe coherence.

\underline{Case $(i)(b)$} $\enhancedStep[n][n+1][n-1]$ and $\enhancedStep[n-1][n][n+1]$ are mutually coherent.
If $\enhancedStep\in\ S(\c L')$ satisfies $n-1 < j$, then $j = n$, $a = n-1$, so $\enhancedStep$ is mutually coherent with $\enhancedStep[n][n+1][n-1]$ and $\enhancedStep[n-1][n][n+1]$.

\underline{Case $(i)(c)$ and $(iii)(e)$} $\enhancedStep[x][n+1][y]$ and $\enhancedStep[y][n][n+1]$ are mutually coherent, and if $\enhancedStep\in\ S'$ then $j \leq y$, so $\enhancedStep$ is mutually coherent with both $\enhancedStep[x][n+1][y]$ and $\enhancedStep[y][n][n+1]$.

\underline{Case $(ii)(b)$ and $(iii)(b)$} Changing the endpoint of the last enhanced step doesn't interfere with mutual coherence (with previous steps).

\underline{Case $(ii)(c)$ and $(ii)(d)$} For $p \in [m-1]$, $\enhancedStep[x][n+1][y_1]$ and $\enhancedStep[y_p][y_{p+1}][n+1]$ are mutually coherent as $\enhancedStep[x][n][y_1]$ and $\enhancedStep[y_p][y_{p+1}][n]$ are in $\c S(\c L')$; if $\enhancedStep\in \c S'$, then $j \leq \max\{x,y_1\} \leq n-1$ and thus $\enhancedStep$ is mutually coherent with $\enhancedStep[y_p][y_{p+1}][n+1]$, and mutually coherent with $\enhancedStep[x][n+1][y_1]$ as $\enhancedStep[x][n][y_1]\in \c S(\c L')$.

\underline{Case $(iii)(d)$} The above argument applies here, replacing $y_1$ by $y$.

\underline{Case $(iii)(c)$} The steps $\enhancedStep[x][n][y]$, $\enhancedStep[n][n+1][y]$ and $\enhancedStep[y][n][n+1]$ are mutually coherent, and if $\enhancedStep\in \c S(\c L')\ssm\{\enhancedStep[y][n-1][n]\}$, then the above argument again applies.

\vspace{0.2cm}

Finally, we prove that there exists no other coherent paths of size $n+1$.

\underline{Case $(i)$} If the last step of $\c L'$ is $\enhancedStep[x][n][n-1]$, then the 3 claimed $\c L$ are the only diagonal-avoiding lattice paths whose restriction is $\c L'$ (as lattice paths must be North-East increasing).

\underline{Case $(ii)$} If $\c L$ restrict to $\c L'$ whose last steps are $\enhancedStep[x][n][y_1]$, $\enhancedStep[y_1][y_2][n],\dots, \enhancedStep[y_{m-1}][y_m][n]$ with $x < n-1$, $m \geq 3$ and $y_1 < \dots < y_m = n-1$, then consider the last step of the form $\enhancedStep[i][n+1]$ in $\c L$.
Either $a = n$ and $\c L$ is necessarily in cases $(a)$ or $(b)$;
or $(i,a) \in \{(x,y_1)\}\cup\{(n,y_p)\}_{p\in[m]}$.
The first possibility leads necessarily to cases $(c)$ and $(d)$, while the latest lead to non-coherent paths, as $\enhancedStep[x][n][y_1]$ would be in $\c L$ and is not mutually coherent with $\enhancedStep[y_p][y_{p+1}][n+1]$ for $p \geq 2$.

\underline{Case $(iii)$}  If the last steps of $\c L'$ are $\enhancedStep[x][n][n-1]$ and $\enhancedStep[y][n-1][n]$, then there are 6 diagonal-avoiding lattice paths whose restriction is $\c L'$ (as lattice paths must be North-East increasing). The only non-coherent one is given by $\c S(\c L) = \bigl(\c S(\c L') \ssm \{\enhancedStep[y][n-1][n]\} \bigr) \cup \{\enhancedStep[n][n+1][y], ~\enhancedStep[y][n-1][n+1], ~\enhancedStep[n-1][n][n+1]\}$, which is not coherent as $\enhancedStep[n][n+1][y]$ and $\enhancedStep[n-1][n][n+1]$ are not mutually coherent (as $y < n-1$).
\end{proof}

Now, we know how to inductively construct all coherent lattice paths: it is time to prove the reciprocal of \Cref{thm:EnhancedStepsCriterion}.
The proof is cumbersome but not difficult: for each of the $12$ cases of \Cref{thm:InductionProcess}, we exhibit a vector $\b\omega$ that captures it.

\begin{theorem}\label{thm:CoherenceIsCoherence}
Coherent paths on $\HypSimplTwo$ are in bijection with coherent lattice paths of size $n$.
\end{theorem}

\begin{proof}[Proof of \Cref{thm:CoherenceIsCoherence}]
\Cref{thm:EnhancedStepsCriterion} proves that the application $\scr L$ sends injectively coherent paths on $\HypSimplTwo$ to coherent lattice paths of size $n$.
We now prove the converse: if $\c L$ is a coherent lattice path of size $n$, then there exists a coherent path $P$ on $\HypSimplTwo$ such that $\scr L(P) = \c L$.
We use the induction process of \Cref{thm:InductionProcess}.
Thanks to \Cref{exmp:SmallDApaths}, we know that all coherent lattice paths of size $3$ and $4$ are coherent paths on $\HypSimplTwo[3]$ and $\HypSimplTwo[4]$.
Next, we prove that if $\c L' = \scr L(P')$ for~$\c L'$ a coherent lattice path of size $n$, then for all coherent path $\c L$ of size $n+1$ which restrict to $\c L'$ (in all 12 cases of \Cref{thm:InductionProcess}), there is a coherent path $P$ on $\HypSimplTwo[n+1]$ such that $\c L = \scr L(P)$.

Let $\c L'$ be a coherent lattice path of size $n$ such that $\c L' = \scr L(P')$ where $P'$ is a coherent path and $\b\omega'\in\R^n$ captures $P'$.
We are going to find $\omega_{n+1}$ such that $\b\omega := (\omega'_1, \dots, \omega'_n, \omega_{n+1})\in\R^{n+1}$ captures a path $P$ with $\scr L(P) = \c L$ (in some cases, we will also modify $\omega_n$ slightly).
We denote $\slope(\step) = \frac{\omega_j-\omega_i}{c_j-c_i}$ as usual.
As we will focus the behavior of the points $(c_i+c_j, \omega_i+\omega+_j)$ for $(i,j)\in\c L$, 
in order to ease notation, we say “the point $(i,j)$" instead of “the point $(c_i+c_j, \omega_i+\omega_j)$".
We will distinguish three cases following the main cases of \Cref{thm:InductionProcess}.

\underline{Case $(i)$} Suppose the last step of $\c L'$ is $\enhancedStep[x][n][n-1]$ with $x < n-1$.
Remark that for $i\notin\{x,n-1,n\}$:\vspace{-0,2cm}
$$\slope(\step[n-1][n]) < \slope(\step[x][n]) < \slope(\step[i][n])$$

\vspace{-0,2cm}Indeed, the first inequality comes directly from the last step of $\c L'$: if the inequality were reversed, the last step would have been $\enhancedStep[n-1][n][x]$ instead.
The second inequality follows from \Cref{lem:ConvexityLemma} applied to the triangle $(x,n-1)$, $(i,n-1)$, $(n,n-1)$.

Note first that, for $i < n$, if $\omega_{n+1}$ satisfies that $\slope(\step[n][n+1]) < \slope(\step[i][n])$, then there can not be a step $\enhancedStep[i][n+1][a]$ in $\c S(P)$ as $\slope(\step[i][n+1]) < \slope(\step[i][n])$ by \Cref{lem:ConvexityLemma}, thus the points associated to $(i,n+1)$ are all below the path $P'$ and do not belong to $P$.
As $\slope(\step[n][n+1])$ is a continuous increasing function of $\omega_{n+1}$, this gives three regimes.

\emph{Case $(i)(a)$:}
To obtain $P$ with $\scr L(P)$ of form $(i)(a)$ in \Cref{thm:InductionProcess}, 
choose $\omega_{n+1}$ small enough to get $\slope(\step[n][n+1]) < \slope(\step[n-1][n])$, then the path $P$ captured by $\b \omega$ has $\c S(P')\subset \c S(P)$ by the above, and $\enhancedStep[n-1][n+1][n]\in \c S(P)$ by \Cref{lem:ConvexityLemma} on the triangle $(n,n-1)$, $(n-1,n+1)$, $(n,n+1)$.

\emph{Case $(i)(b)$:}
To obtain $P$ with $\scr L(P)$ of form $(i)(b)$ in \Cref{thm:InductionProcess},
choose $\omega_{n+1}$ to satisfy $\slope(\step[n-1][n]) < \slope(\step[n][n+1]) < \slope(\step[x][n])$, then the path $P$ captured by $\b \omega$ has  $\c S(P')\subset \c S(P)$ by the above,
and $\enhancedStep[n][n+1][n-1]\in \c S(P)$ and $\enhancedStep[n-1][n][n+1]\in \c S(P)$ by applying \Cref{lem:ConvexityLemma} to the triangle $(n,n-1)$, $(n-1,n+1)$, $(n,n+1)$.

\emph{Case $(i)(c)$:}
To obtain $P$ with $\scr L(P)$ of form $(i)(c)$ in \Cref{thm:InductionProcess},
choose $\omega_{n+1}$ to satisfy $\slope(\step[x][n]) < \slope(\step[n][n+1]) < \min_{i\notin\{x,n\}}\slope(\step[i][n])$, then the path $P$ captured by $\b \omega$ has $\c S(P')\ssm\{\enhancedStep[x][n][n-1]\}\subset \c S(P)$ by the above,
and $\enhancedStep[x][n+1][n-1]\in \c S(P)$ and $\enhancedStep[n-1][n][n+1]\in \c S(P)$.

We have shown that if $\c L'$ is in the case $(i)$, then all the paths $\c L$ that restrict to $\c L'$ are of the form $\c L = \scr L(P)$ for some coherent $P$.

\underline{Case $(ii)$} Suppose the last steps of $\c L'$ are $\enhancedStep[x][n][y_1], \enhancedStep[y_1][y_2][n], \dots, \enhancedStep[y_{m-1}][n-1][n]$ with $m\geq3$. Then for all $i\notin \{x,y_1,n\}$, by the same argument as before:\vspace{-0.2cm}
$$\slope(\step[y_{m-1}][n-1]) < \slope(\step[y_{m-2}][y_{m-1}]) < \dots < \slope(\step[y_1][y_2]) <  \slope(\step[x][n]) < \slope(\step[i][n])$$

We distinguish three regimes.

\emph{Case $(ii)(a)$:}
To obtain $P$ with $\scr L(P)$ of form $(ii)(a)$ in \Cref{thm:InductionProcess},
choose $\omega_{n+1}$ small enough to satisfy $\slope(\step[n][n+1]) < \slope(\step[y_{m-1}][n-1])$, then the path $P$ captured by $\b\omega$ has $\c S(P')\subset \c S(P)$, and $\enhancedStep[n-1][n+1][n]\in\c S(P)$ by \Cref{lem:ConvexityLemma} on the triangle $(n,n-1)$, $(n-1,n+1)$ and $(n,n+1)$.

\emph{Case $(ii)(b)$:}
To obtain $P$ with $\scr L(P)$ of form $(ii)(b)$ in \Cref{thm:InductionProcess},
choose $\omega_{n+1}$ to satisfy $\slope(\step[y_{m-1}][n-1]) < \slope(\step[n][n+1]) < \slope(\step[y_{m-2}][y_{m-1}])$, then the path $P$ captured by $\b\omega$ has $\c S(P')\ssm\{\enhancedStep[y_{m-1}][n-1][n]\}\subset \c S(P)$, 
and $\enhancedStep[y_{m-1}][n+1][n]\in\c S(P)$ because applying \Cref{lem:ConvexityLemma} to the triangle $(y_{m-1},n)$, $(n,n-1)$ and $(n-1,n+1)$ gives that $\slope(\step[n-1][n+1]) < \slope(\step[n-1][n])$,
and applying it to $(y_{m-1},n)$, $(y_{m-1},n+1)$ and $(n,n+1)$ gives that $\slope(\step[y_{m-1}][n+1]) > \slope(\step[n][n+1])$.

\emph{Case $(ii)(d)$:}
To obtain $P$ with $\scr L(P)$ of form $(ii)(d)$ in \Cref{thm:InductionProcess},
choose $\omega_{n+1}$ to satisfy $\slope(\step[x][n]) < \slope(\step[n][n+1]) < \min_{i\notin\{x,n\}}\slope(\step[i][n])$, then the path $P$ captured by $\b\omega$ has $\c S'\subset \c S(P)$,
and by applying \Cref{lem:ConvexityLemma} to the triangle $(x,y_1)$, $(n,y_1)$ and $(n+1,y_1)$, one gets that $\enhancedStep[x][n+1][y_1]\in\c S(P)$.
Moreover, the projected path $\bigl((n+1,y_1),(n+1,y_2), \dots, (n+1,y_{m-1}),(n+1,n-1)\bigr)$ is parallel and higher than the projected path $\bigl((n,y_1),(n,y_2), \dots, (n,y_{m-1}),(n,n-1)\bigr)$, thus $\enhancedStep[y_i][y_{i+1}][n+1]\in\c S(P)$ for $i\in [1,m-2]$.
As $\slope(\step[y_{m-1}][n-1]) < \slope(\step[x][n]) \leq \slope(\step[n-1][n])$ in $P'$, \Cref{lem:ConvexityLemma} ensures that $\enhancedStep[y_{m-1}][n][n+1] \in \c S(P)$.

\emph{Case $(ii)(c)$:}
To obtain $P$ with $\scr L(P)$ of form $(ii)(c)$ in \Cref{thm:InductionProcess},
note that in the previous sub-case there is no point $(i,n)$ in $P$ except from $(n,n+1)$.
So lowering the value of $\omega_n$ (with the same fixed $\omega_{n+1}$ as in the previous sub-case) will not affect the path except in the last triangle $(y_{m-1},n+1)$, $(n-1,n+1)$, $(n,n+1)$.
Taking $\omega_n$ low enough to satisfy $\slope(\step[y_{m-1}][n-1]) > \slope(\step[n-1][n])$, we obtain a path $\Tilde{P}$ with $\c S(P)\ssm\{\enhancedStep[y_{m-1}][n][n+1]\}\subset \c S(\Tilde{P})$ and $\{\enhancedStep[y_{m-1}][n-1][n+1],~ \enhancedStep[n-1][n][n+1]\}\in\c S(\Tilde{P})$.

\underline{Case $(iii)$} Suppose that the last steps of $\c L'$ are $\enhancedStep[x][y][n]$, $\enhancedStep[y][n-1][n]$. Then for $i\notin\{x,y,n\}$:\vspace{-0.2cm}
$$\slope(\step[y][n-1]) < \slope(\step[y][n]) < \slope(\step[x][n]) < \slope(\step[i][n])$$

\vspace{-0.2cm}Indeed, $\slope(\step[y][n]) < \slope(\step[x][n])$ otherwise $\enhancedStep[x][n][y]\notin\c S(P')$, and $\slope(\step[y][n-1]) < \slope(\step[y][n])$ as already $\slope(\step[y][n-1]) < \slope(\step[n-1][n])$.

\emph{Case $(iii)(a)$:}
To obtain $P$ with $\scr L(P)$ of form $(iii)(a)$ in \Cref{thm:InductionProcess},
choose $\omega_{n+1}$ small enough to satisfy $\slope(\step[n][n+1]) < \slope(\step[y][n-1])$, then $\c S(P')\subset \c S(P)$ and $\enhancedStep[n-1][n+1][n]\in\c S(P)$.

\emph{Case $(iii)(b)$:}
To obtain $P$ with $\scr L(P)$ of form $(iii)(b)$ in \Cref{thm:InductionProcess},
choose $\omega_{n+1}$ to satisfy $\slope(\step[y][n-1]) < \slope(\step[n][n+1]) < \slope(\step[y][n])$, then $\c S(P')\ssm\{\enhancedStep[y][n-1][n]\}\subset \c S(P)$,
and as $\slope(\step[n][n+1]) < \slope(\step[y][n])$, \Cref{lem:ConvexityLemma} ensures $\enhancedStep[y][n+1][n]\in\c S(P)$.

\emph{Case $(iii)(c)$:}
To obtain $P$ with $\scr L(P)$ of form $(iii)(c)$ in \Cref{thm:InductionProcess},
choose $\omega_{n+1}$ to satisfy $\slope(\step[y][n]) < \slope(\step[n][n+1]) < \slope(\step[x][n])$, then $\c S(P')\ssm\{\enhancedStep[y][n-1][n]\}\subset \c S(P)$,
and as $\slope(\step[n][n+1]) > \slope(\step[y][n])$, \Cref{lem:ConvexityLemma} ensures $\{\enhancedStep[n][n+1][y], ~\enhancedStep[y][n][n+1]\}\subset\c S(P)$.

\emph{Case $(iii)(e)$:}
To obtain $P$ with $\scr L(P)$ of form $(iii)(e)$ in \Cref{thm:InductionProcess},
choose $\omega_{n+1}$ to satisfy $\slope(\step[x][n]) < \slope(\step[n][n+1]) < \min_{i\notin\{x,n\}} \slope(\step[i][n])$, then $\c S'\cup\{\enhancedStep[x][n+1][y]\}\subset \c S(P)$,
and $\enhancedStep[y][n+1][n]\in\c S(P)$ as $\slope(\step[y][n-1]) < \slope(\step[n-1][n])$.

\emph{Case $(iii)(d)$:}
To obtain $P$ with $\scr L(P)$ of form $(iii)(d)$ in \Cref{thm:InductionProcess},
from the previous value of $\omega_{n+1}$, we lower the value of $\omega_n$ until $\slope(\step[y][n-1]) > \slope(\step[n-1][n])$.
As no other point of the form $(i,n)$ belongs to $P$, this new $\b\omega$ captures a path $\Tilde{P}$ with $\c S(P) \ssm\{\enhancedStep[y][n][n+1]\} \subset\c S(\Tilde{P})$ and $\{\enhancedStep[y][n-1][n+1], ~\enhancedStep[n-1][n][n+1]\}\subset \c S(\Tilde{P})$.

In all 12 cases, we proved that if the restriction of $\c L$ is the image by $\scr L$ of a coherent path on $\HypSimplTwo[n-1]$, then $\c L$ is the image by $\scr L$ of a coherent path on $\HypSimplTwo$.
Thus $\scr L$ is surjective.
\end{proof}

\begin{corollary}\label{cor:EnhancedStepCriterionIsCoherence}
For any generic direction $\b c$, a $\b c$-monotone path $P$ on $\HypSimplTwo$ is coherent if and only if it satisfies the enhanced steps criterion:
for all couples of enhanced steps $\enhancedStep \prec \enhancedStep[x][y][z] \in \c S(P)$ with $x < j$, one has $j = z$ or $x = a$.
\end{corollary}

\begin{proof}[Proof of \Cref{cor:EnhancedStepCriterionIsCoherence}]
Combine \Cref{thm:CoherenceIsCoherence} (for ``if'') and \Cref{thm:EnhancedStepsCriterion} (for ``only if'').
\end{proof}

\begin{remark}\label{rmk:NormalFanOfMn2}
From the details of the proof, the reader can retrieve the inductive construction of the normal fan of the monotone path polytope $\MPPHypSimplTwo$.
Indeed, for each case, we provided the inequalities that $\b \omega$ shall respect for the path of size $n+1$, in term of the inequalities that shall be respected for its restriction (and then we created such an $\b\omega$).

Yet, this neither seems to yield an easy description of the face lattice of the monotone path polytope $\MPPHypSimplTwo$, nor a clear way to count the number of faces of any dimension apart from the vertices.
If one wants to compute the normal fan of the monotone path polytope $\MPPHypSimplTwo$, another (and maybe better) way to do so it to use \Cref{prop:CaptureCriterion} directly: the combinatorics of a monotone path is directly associated to a collection of inequalities, forming a polyhedral cone, and coherent monotone paths are exactly the ones for which this cone is non-empty.
Finding the irredundant inequalities among this collection remains a difficult task.
\end{remark}

\newpage
\section{Counting the number of coherent monotone paths on $\HypSimplTwo$}\label{ssec:NumberVertMPPHypSimplTwo}
\subsection{Counting all coherent paths}\label{ssec:CountingTotal}

The induction process of \Cref{thm:InductionProcess} allows us to count precisely the number of coherent lattice paths, which is the number of vertices of $\MPPHypSimplTwo$ by \Cref{thm:CoherenceIsCoherence}.
We gather the coherent lattice paths of size $n$ into 3 groups, matching the 3 main cases of \Cref{thm:InductionProcess}: cases $(i)$, $(ii)$, $(iii)$.

Let $t_n$ (respectively $q_n$ ; respectively $c_n$) be the number\footnote{The letters ``$t$'', ``$q$'' and ``$c$'' stand for \emph{trois}, \emph{quatre} and \emph{cinq} (3, 4 and 5 in French)} of coherent paths $\c L$ of size $n$ of type~$(i)$, that is to say whose last step is $\enhancedStep[x][n][n-1]$ (respectively of type $(ii)$, that is to say whose last steps are $\enhancedStep[x][n][y_1],\enhancedStep[y_1][y_2][n],\dots,\enhancedStep[y_{m-1}][n-1][n]$ with $m\geq3$ ; respectively of type $(iii)$, that is to say whose last steps are $\enhancedStep[x][n][y], \enhancedStep[y][n-1][n]$).
The induction process of \Cref{thm:InductionProcess} gives:

\begin{proposition}\label{prop:RecursiveFormula}
The sequences $t_n$, $q_n$ and $c_n$ satisfy the following recursive formula:
$$\text{for all } n\geq 4,~ \begin{pmatrix}t_{n+1} \\ q_{n+1} \\ c_{n+1}\end{pmatrix} = M\begin{pmatrix}t_n \\ q_n \\ c_n\end{pmatrix} ~~~~~\text{with}~~~~ M = \begin{pmatrix}
    1&2&2 \\ 0&2&1 \\ 2&0&2
\end{pmatrix} ~~~~\text{and}~~~~~ \begin{pmatrix}t_4 \\ q_4
\\ c_4\end{pmatrix} = \begin{pmatrix} 3 \\ 1 \\ 4 \end{pmatrix}.$$
\end{proposition}

\begin{proof}
The values for $t_4$, $q_4$ and $c_4$ follow from \Cref{exmp:SmallDApaths}.

Looking at the induction process in \Cref{thm:InductionProcess}, for each case $(i)(a)$ to $(iii)(e)$, one can identify if the created coherent path of size $n+1$ is of the type of case $(i)$, $(ii)$ or $(iii)$.
For example, if $\c L'$ of size $n$ ends by a step $\enhancedStep[x][n][n-1]$, then there are three $\c L$ of size $n+1$ that restrict to $\c L'$: in case $(i)(a)$, $\c L$ ends with $\enhancedStep[y][n+1][n]$ (with $y = n-1$) so it belongs to type $(i)$.
The case analysis is summarized in the following table (reading off the table gives the matrix $M$):

\begin{center}
\begin{tabular}{c|ccccc}
     & $(a)$ & $(b)$ & $(c)$ & $(d)$ & $(e)$ \\ \hline
     $(i)$ & $(i)$ & $(iii)$ & $(iii)$ & & \\
     $(ii)$ & $(i)$ & $(i)$ & $(ii)$ & $(ii)$ & \\
     $(iii)$ & $(i)$ & $(i)$ & $(iii)$ & $(ii)$ & $(iii)$ \\
\end{tabular}
\end{center} \vspace{-0.66cm}\qedhere
\end{proof}

\begin{theorem}\label{thm:NumberVertsMPPHypSimplTwo}
For $n\geq 4$, there are $\frac{1}{3}\bigl(25\cdot 4^{n-4} - 1\bigr)$ coherent paths of size $n$.
\end{theorem}

\begin{proof}
By definition, the total number of coherent paths of size $n$ is $t_n+q_n+c_n$.

A quick analysis of $M$ shows that $\text{Sp}(M) = \{0,1,4\}$ with $(2,-2,-1)M = (0,0,0)$.
Thus for all $n$: $2t_n = 2q_n + c_n$.
It follows that if $c_n = t_n+1$, then $t_n = 2q_n+1$, and thus
$c_{n+1} = 2t_n+2c_n = t_n+2q_n+2c_n+1 = t_{n+1}+1$.
By induction: for $n\geq 4$, $c_n = t_n+1$.
Hence $t_{n+1}+q_{n+1}+c_{n+1} = 4(t_n+q_n+c_n) +1$.
Using $t_4+q_4+c_4 = 8$, we conclude.
\end{proof}

\begin{remark}\label{rmk:0percentOfMonotonePaths}
One may wonder if the number of \textit{coherent} monotone paths on $\HypSimplTwo$ is ``big'' or not.
It grows exponentially according to the dimension (\ie as $4^n$), but we would like to compare it to the total number of monotone paths on $\HypSimplTwo$.
The latter is hard to compute, so we instead give a lower bound, thanks to our induction process: we want a cheap estimate for the total number of diagonal-avoiding paths.

As we want to keep it concise, we invite the reader to do its own drawings.
Let $d_n$ be the number of diagonal-avoiding paths such that the last enhanced step is $\enhancedStep[x][n][n-1]$, and let $s_n$ be the number of diagonal-avoiding paths that do not\footnote{The letters ``$d$'' and ``$s$'' stand for \emph{Drei} and \emph{Sechs} (3 and 6 in German).}.
Mimicking the induction process of \Cref{thm:InductionProcess}(i), it is easy to see that for each path finishing by $\enhancedStep[x][n][n-1]$, there are exactly 3 paths of size $n+1$ which restrict to it: 1 of them finishing by $\enhancedStep[y][n+1][n]$.
On the other hand, for each path that does not finish by $\enhancedStep[x][n][n-1]$, there are \textbf{at least} 6 paths of size $n+1$ which restrict to it: 2 of them finishing by $\enhancedStep[y][n+1][n]$. The 4 others are obtained by transforming all $\enhancedStep[y_i][y_{i+1}][n]$ into $\enhancedStep[y_i][y_{i+1}][n+1]$ and choosing whether to include the lattice points $(n, x)$ and $(n+1, n-1)$.

We are far from counting all the diagonal-avoiding paths (for instance, the $10^{\text{th}}$ path on the bottom right of \Cref{fig:SmallDApaths} is not counted), but it is enough.
This proves $d_{n+1} = d_n + 2 s_n$ and $s_{n+1} \geq 2 d_n + 4 s_n$.
Thus $s_{n+1} \geq 2 d_{n+1}$, hence $d_{n+1} \geq d_n + 2(2d_n) = 5 d_n$.
Finally, the total number of monotone paths on $\HypSimplTwo$ grows faster than $5^n$ (as already $d_n$ does): asymptotically, the number of coherent monotone paths is negligible within the total number of monotone paths.

\vspace{0,15cm}

Interestingly two behaviors seem to emerge.
In some cases, for families of polytopes like cubes and simplices \cite{BilleraSturmfels-FiberPolytope}, and matroid polytopes (for non-generic directions $\b e_S$) \cite[Theorem 8.1]{BlackSanyal-FlagPolymatroids}, all monotone paths are coherent.
On the other side, for cyclic polytopes \cite[Remark 3.6]{DeLoeraRambauSantos}, cross-polytopes \cite[last paragraph]{BlackDeLoera2021monotone}, standard permutahedra \cite[Theorem 1.2]{AngelGorinHolroyd-PatternTheorem}, and for $\HypSimplTwo$, the ratio between number of coherent monotone paths and the number of monotone paths goes to $0$ as the dimension grows (in most cases, exponentially fast).
It would be intriguing to have a family of polytopes (endowed with some combinatorics) such that the number of coherent paths represent a non-zero fraction of the number of monotone paths, if the dimension grows.
\end{remark}

\subsection{Counting according to length}\label{ssec:CountingByLengths}

The length of (coherent) monotone paths, that is to say the number of vertices in the path, is a crucial parameter.
On the one hand, in linear programming, the lengths of coherent monotone paths govern the complexity of the shadow vertex rule for the simplex algorithm.
On the other hand, the (non-coherent) monotone paths of maximal length on the hypersimplex $\HypSimpl$ are in bijection with standard Young tableaux of the rectangular shape
$k \times (n - k)$, as pointed out by Postnikov \cite[Example 10.4]{Postnikov-PositiveGrassmannianAndPolyhedralSubdivisons}.
In particular, there are $\frac{1}{n-1}\binom{2(n-2)}{n-2}$ longest monotone paths on $\HypSimplTwo$, with length $2(n-2)$.
These paths form a Catalan family: they are naturally in bijection with Catalan objects (Dyck paths, binary trees, standard Young tableaux of shape $2 \times (n-2)$, triangulations of a $n$-gon...).
The bijection with Dyck paths can easily be seen once the reader realizes that the longest lattice paths has steps of size 1, \ie $\enhancedStep$ with $j = i+1$, and hence stays under that diagonal: turning his or her head by $135^{\circ}$, the reader will see the Dyck path.

However, the paths of length $2(n-2)$ on $\HypSimplTwo$ cannot be coherent.
As we will see in \Cref{thm:LongestCoherentPath}, the maximum length of a coherent monontone path on $\HypSimplTwo$ is $\left\lfloor\frac{3}{2}(n-1)\right\rfloor$.
Indeed, with the formula of \Cref{thm:NumberVertsMPPHypSimplTwo} we solved the question we started with (determining the vertices of $\MPPHypSimplTwo$ and count them), but notwithstanding, one can go even further in the analysis of the induction process.
Let $t_{n,\ell}$ be the number of coherent paths of size $n$ and length $\ell$ that end with a step $\enhancedStep[x][n][n-1]$ 
and let $q_{n,\ell}$, $c_{n,\ell}$ be the counterparts for the two other main cases of the induction.
Let $T_n = \sum_{\ell} t_{n,\ell} z^\ell$, $Q_n = \sum_{\ell} q_{n,\ell} z^\ell$ and $C_n = \sum_{\ell} c_{n,\ell} z^\ell$ be the associated generating polynomials, then \Cref{thm:InductionProcess} gives the following:

\begin{proposition}\label{prop:RecursiveFormulaLength}
The sequences of polynomials $T_n$, $Q_n$ and $C_n$ satisfy the recursive formula:
$$\text{for } n\geq 4,~ \begin{pmatrix}T_{n+1} \\ Q_{n+1} \\ C_{n+1}\end{pmatrix} = \c M\begin{pmatrix}T_n \\ Q_n \\ C_n\end{pmatrix} ~\text{with}~~ \c M = \begin{pmatrix}
    z&1+z&1+z \\ 0&1+z&z \\ z+z^2&0&1+z
\end{pmatrix} ,~ \begin{pmatrix}T_4 \\ Q_4
\\ C_4\end{pmatrix} = \begin{pmatrix} z^4+2z^3 \\ z^4 \\ 2z^4+2z^3 \end{pmatrix}.$$
\end{proposition}

\begin{remark}
Note that evaluating the previous relation at $z = 1$ gives back \Cref{prop:RecursiveFormula}.
\end{remark}

\begin{proof}[Proof of \Cref{prop:RecursiveFormulaLength}]
The values for $T_4$, $Q_4$ and $C_4$ have been explored in \Cref{exmp:SmallDApaths}.

Looking at the induction process, for each case $(i)(a)$ to $(iii)(e)$, one can identify the length of the created coherent path of size $n+1$.
For example, if $\c L'$ of size $n$ and length $\ell$ ends by a step $\enhancedStep[x][n][n-1]$, then there are three $\c L$ of size $n+1$ that restricts to $\c L'$: in case $(i)(a)$, $\c L$ contains 1 step more than $\c L'$ so it has length $\ell +1$.
The case analysis is summarized in the following table, assuming the restricted path is of length $\ell$:
\begin{center}
\begin{tabular}{c|ccccc}
     & $(a)$ & $(b)$ & $(c)$ & $(d)$ & $(e)$ \\ \hline
     $(i)$ & $\ell+1$ & $\ell+2$ & $\ell+1$ & & \\
     $(ii)$ & $\ell+1$ & $\ell$ & $\ell+1$ & $\ell$ & \\
     $(iii)$ & $\ell+1$ & $\ell$ & $\ell+1$ & $\ell+1$ & $\ell$ \\
\end{tabular}
\end{center}

Reading off this table together with the one of the proof of \Cref{prop:RecursiveFormula} yields $\c M$.
\end{proof}

The matrix $\c M$ (over the polynomial ring) has three eigenvalues $\lambda_0 = 0$, $\lambda_+ = 1+\frac{3}{2}z + \frac{1}{2}z\sqrt{4z+5}$, and $\lambda_- = 1+\frac{3}{2}z - \frac{1}{2}z\sqrt{4z+5}$ with associated (left) eigenvectors:
$$\b x_0 = \begin{pmatrix}
    -1 & 1 & \frac{1}{1+z}
\end{pmatrix}, ~~
\b x_+ = \begin{pmatrix}
    1 & \frac{\sqrt{4z+5}-1}{2z} & \frac{z\sqrt{4z+5}+z+2}{2(z^2+z)}
\end{pmatrix}, ~~
\b x_- = \begin{pmatrix}
    1 & -\frac{\sqrt{4z+5}+1}{2z} & \frac{-z\sqrt{4z+5}+z+2}{2(z^2+z)}
\end{pmatrix}$$

Unfortunately, the square roots in the eigenvalues and eigenvectors make it very difficult to derive an explicit formula as simple as in \Cref{thm:NumberVertsMPPHypSimplTwo}, but we can prove two very interesting properties on the number of coherent paths of a given length.

\begin{theorem}\label{thm:LongestCoherentPath}
For a fixed size $n$ with $n\geq4$, the longest coherent path of size $n$ is of length $\ell_{\max} = \left\lfloor\frac{3}{2}(n-1)\right\rfloor$.
The number of coherent paths of size $n$ and length $\ell_{\max}$ is $1$ if $n$ is odd, and $\left\lfloor\frac{3}{2}(n-1)\right\rfloor$ if $n$ is even.
\end{theorem}

\begin{proof}
We will prove by induction the slightly stronger following statement on the degrees and leading coefficients of $T_n$, $Q_n$ and $C_n$.
Denote $\nu_n = \left\lfloor\frac{3}{2}(n-1)\right\rfloor$:
$$\left\{\begin{array}{llll}
    \text{if } n \text{ odd,} & T_n = (\nu_n-2)z^{\nu_n-1} + o(z^{\nu_n-1}), & Q_n = O(z^{\nu_n-1}), & C_n = z^{\nu_n} + o(z^{\nu_n}) \\
    \text{if } n \text{ even,} & T_n = z^{\nu_n} + o(z^{\nu_n}), & Q_n =   z^{\nu_n}+o(z^{\nu_n}), & C_n = (\nu_n-2)z^{\nu_n} + o(z^{\nu_n})
\end{array}\right.$$

This statement holds for $n = 4$ as $\nu_4 = 4$, $T_4 = z^4 + o(z^4)$, $Q_4 = z^4$ and $C_4 = 2z^4+o(z^4)$.

Now, it is just a matter of multiplying by $\c M$.
Suppose $n$ is odd and the statement holds, then:
$$\begin{array}{rll}
     T_{n+1} &= zT_n + (1+z)Q_n+(1+z)C_n  & \\
     &= O(z^{\nu_n}) + O(z^{\nu_n}) + z^{\nu_n+1} + o(z^{\nu_n+1}) & \\
     &= z^{\nu_{n+1}} + o(z^{\nu_{n+1}}) &\text{as } \nu_{n+1} = \nu_n+1 \\ \\
     Q_{n+1} &= (1+z)Q_n+zC_n  & \\
     &= O(z^{\nu_n}) + z^{\nu_n+1} + o(z^{\nu_n+1}) & \\
     &= z^{\nu_{n+1}} + o(z^{\nu_{n+1}}) &\text{as } \nu_{n+1} = \nu_n+1 \\ \\
     C_{n+1} &= (z^2+z)T_n + (1+z)C_n  & \\
     &= (\nu_n-2)z^{\nu_n+1} +o(z^{\nu_n+1}) + z^{\nu_n+1} + o(z^{\nu_n+1}) & \\
     &= (\nu_{n+1}-2)z^{\nu_{n+1}} + o(z^{\nu_{n+1}}) &\text{as } \nu_{n+1} = \nu_n+1
\end{array}$$

Suppose $n$ is even, and the statement holds, then:
$$\begin{array}{rll}
     T_{n+1} &= zT_n + (1+z)Q_n+(1+z)C_n  & \\
     &= z^{\nu_n+1} + z^{\nu_n+1} + (\nu_n-2)z^{\nu_n+1} + o(z^{\nu_n+1}) & \\
     &= (\nu_{n+1}-2)z^{\nu_{n+1}-1} + o(z^{\nu_{n+1}-1}) &\text{as } \nu_{n+1} = \nu_n+2 \\ \\
     Q_{n+1} &= (1+z)Q_n+zC_n  & \\
     &= O(z^{\nu_n+1}) + O(z^{\nu_n+1}) & \\
     &= O(z^{\nu_{n+1}-1}) &\text{as } \nu_{n+1} = \nu_n+2 \\ \\
     C_{n+1} &= (z^2+z)T_n + (1+z)C_n  & \\
     &= z^{\nu_n+2} + o(z^{\nu_n+2}) + O(z^{\nu_n+1}) & \\
     &= z^{\nu_{n+1}} + o(z^{\nu_{n+1}}) &\text{as } \nu_{n+1} = \nu_n+2
\end{array}$$

By induction $T_n+Q_n+C_n$ has degree $\nu_n$, and leading coefficient $1$ if $n$ odd, $\nu_n$ if $n$ even.
\end{proof}

\begin{theorem}\label{thm:PolynomialCoefficient}
For a fixed length $\ell$, the number of coherent paths of size $n \geq \left\lceil \frac{2}{3}\ell +1\right\rceil$ is a polynomial in $n$ of degree $\ell-3$.
\end{theorem}

\begin{proof}
Let $v_{n,\ell}$ be the total number of coherent paths of size $n$ and length $\ell$, then \linebreak$V_n = \sum_{\ell} v_{n,\ell} z^{\ell} = T_n+Q_n+C_n$.
We can compute $V_n$ thanks to the powers of $\c M$:
$$V_{n+4} = \begin{pmatrix}
    1&1&1
\end{pmatrix}\c M^n\begin{pmatrix}
    T_4 \\ Q_4 \\ C_4
\end{pmatrix}$$

With the eigenvalues and eigenvectors given above, one can compute:
$$V_{n+4} = \frac{\lambda_+^n - \lambda_-^n}{\sqrt{4z+5}}z^5 + \left(2(\lambda_+^n+\lambda_-^n)+6\frac{\lambda_+^n - \lambda_-^n}{\sqrt{4z+5}}\right)(z^4 + z^3)$$

Note that as $\lambda_+$ and $\lambda_-$ depend on $z$.
Indeed:
$$\frac{\lambda_+^n - \lambda_-^n}{\sqrt{4z+5}} = \sum_k \binom{n}{2k+1}\left(1+\frac{3}{2}z\right)^{n-(2k+1)}\left(\frac{5}{4}+z\right)^k z^{2k+1}$$
and
$$\lambda_+^n + \lambda_-^n = 2\sum_k \binom{n}{2k}\left(1+\frac{3}{2}z\right)^{n-2k}\left(\frac{5}{4}+z\right)^k z^{2k}$$

Not only they are polynomials in $z$ (which was expected as $V_n$ is a polynomial by definition), but we can investigate their coefficients.
It allows us to re-write:
$$V_{n+4} = \sum_{(a,b,c)\in\N^3} \alpha_{n,a,b,c} \left(1+\frac{3}{2}z\right)^{a}\left(\frac{5}{4} + z\right)^{b} z^c$$
where $\alpha_{n,a,b,c}$ is a sum of binomial coefficient $\binom{n}{f(a,b,c)}$ with $f$ a function of $a$, $b$ and $c$.
This coefficient is thus a polynomial in $n$.

By \Cref{thm:LongestCoherentPath}, we know that the polynomial $V_n$ has degree $\left\lfloor \frac{3}{2}(n-1) \right\rfloor$, thus for a fixed~$\ell$, the coefficient of $V_n$ on $z^{\ell}$ is non-zero when $n \geq \left\lceil \frac{2}{3} \ell +1 \right\rceil$.
This coefficient can be seen as (a multiple of) the evaluation at $z = 0$ of the polynomial $\frac{\partial^{\ell}}{\partial z^{\ell}} V_n$.
But this derivative is again a sum of (products of) powers of $\left(1+\frac{3}{2}z\right)$, of $\left(\frac{5}{4} + z\right)$ and of $z$, with no new dependencies in $n$.
Evaluating at $z = 0$ gives that $v_{n,\ell}$ is a sum (which coefficients depend on $\ell$) of $\binom{n}{f(a,b,c)}$: a polynomial in~$n$.
To obtain its degree, we look for the greatest $\kappa$ such that $\binom{n}{\kappa}$ appears in the coefficient of $z^{\ell}$.
Developing of both $\frac{\lambda_+^n - \lambda_-^n}{\sqrt{4z+5}}$ and $\left(\lambda_+^n + \lambda_-^n\right)$, note that $\kappa$ is the power on the factor $z$.
For a fixed $\ell$, the greatest power on $z$ appearing in $\frac{\lambda_+^n - \lambda_-^n}{\sqrt{4z+5}}z^5$ is $\ell - 5$, the greatest in $\left(2(\lambda_+^n+\lambda_-^n)+6\frac{\lambda_+^n - \lambda_-^n}{\sqrt{4z+5}}\right)(z^4 + z^3)$ is $\ell -3$.
Thus, the degree of the polynomial $v_{n,\ell}$, as a polynomial in $n$, is $\ell - 3$.
\end{proof}

\begin{example}
With \Cref{prop:RecursiveFormulaLength}, we compute the number of coherent paths of size $n$, length $\ell$:

\begin{center}
\begin{tabular}{c|c|c|c|c|c|c|c|c|c|c|c|c|c}
    $n\backslash\ell$ &3&4&5&6&7&8&9&10&11&12&13&14&15 \\ \hline
    4 & 4&4&&&&&&&&&&& \\
    5 & 4&16&12&1&&&&&&&&& \\
    6 & 4&28&56&38&7&&&&&&&& \\
    7 & 4&40&132&195&129&32&1&&&&&& \\
    8 & 4&52&240&556&694&448&129&10&&&&& \\
    9 & 4&64&380&1205&2250&2496&1571&501&61&1&&& \\
    10 & 4&76&552&2226&5565&8896&9019&5564&1914&304&13&& \\
    11 & 4&88&756&3703&11627&21416&34622&32725&19881&7236&1375&99&1 \\
\end{tabular}
\end{center}

\vspace{0.5cm}

In this table, one can read out \Cref{thm:LongestCoherentPath} (for $n\leq 11$) by looking at the right-most value in each line.
Furthermore, \Cref{thm:PolynomialCoefficient} ensures that each column $\ell$ is a polynomial in $n$ of degree $\ell-3$.
Observing the rows given, the following holds for $n\geq1$:
\begin{compactenum}
\item[$\bullet$] for $\ell = 3$: $v_{n+3,3} = 4$ is also the number of diagonal-avoiding paths of length $3$.
\item[$\bullet$] for $\ell = 4$: $v_{n+3,4} = 12n-8$ is also the number of diagonal-avoiding paths of length $4$. 
\item[$\bullet$] for $\ell = 5$: $v_{n+4,5} = 4n(4n-1)$ is \textbf{not} the number of diagonal-avoiding paths of length $5$.
\item[$\bullet$] for $\ell = 6$: $v_{n+5,6} = 14n^3-24n^2+11n$. 
\item[$\bullet$] for $\ell = 7$: $v_{n+6,7} = \frac{1}{6}(55n^4 - 2n^3 - 34n^2 + 23n)$.
\end{compactenum}

And one can easily continue this list with a computer, thanks to the matrix recursion of \Cref{prop:RecursiveFormulaLength}.
\end{example}

\section{Open questions}

\paragraph{Computational remarks}
All the objects present in this paper have been implemented with Sage \cite{Sage}.
Namely, we are able to compute monotone path polytopes and label their vertices by the corresponding monotone paths and their normal cones (\ie the cone of $\b \omega$ that captures the monotone path).
Monotone path polytopes are computed as a Minkowski sum of sections, one per vertex, but computing Minkowski sums in high dimensions and with numerous vertices is costly.

\begin{question}
Create and implement more efficient algorithms to construct the monotone path polytope of a linear program $(\polytopeP, \b c)$.
Adapt this method to take into account symmetries of $\polytopeP$.
\end{question}

For $\MPPHypSimplTwo$, numerical statements were checked by $(i)$ constructing the monotone path polytope and counting its vertices (up to dimension $8$); $(ii)$ constructing all monotone paths and solving the linear system to know whether it can be captured or not (up to dimension $9$); $(iii)$ generating all paths that respect the criterion of \Cref{thm:EnhancedStepsCriterion} and verifying if they are coherent (up to dimension $12$); $(iv)$ implementing the matrix recursion (up to dimension $300$).
Fortunately, all these methods lead to the same result!
We implemented similar methods for counting the paths by their length.

Besides, the diagonalization of matrices was done with the help of Sage \cite{Sage} (and later checked by hand and with Wolfram Alpha \cite{WolframAlpha}), which benefits from excellent and easy-to-use tools to deal with matrices over any rings (especially the ring of symbolic expressions). 

\vspace{0,15cm}

\paragraph{Unimodality of the number of coherent monotone paths per length}
We have detailed the behavior of $v_{n,\ell}$ for a fixed length $\ell$.
But on the other side, for a fixed size $n$, one can look at the sequence $\bigl(v_{n,\ell}~;~\ell\in[3,\left\lfloor\frac{3}{2}(n-1)\right\rfloor]\bigr)$.
Jes\'us De Loera enjoined us to look at the properties of these sequences, especially their unimodality and log-concavity.
Quick computations motivate the following:

\begin{conjecture}
For $n\geq 4$, the sequence of numbers $\bigl(v_{n,\ell} ; \ell\in[3,\left\lfloor\frac{3}{2}(n-1)\right\rfloor]\bigr)$, counting the number of coherent monotone paths on $\HypSimplTwo$ according to their lengths, is log-concave (hence unimodal).
\end{conjecture}

\begin{figure}[b]
    \centering
    \includegraphics[scale=0.45]{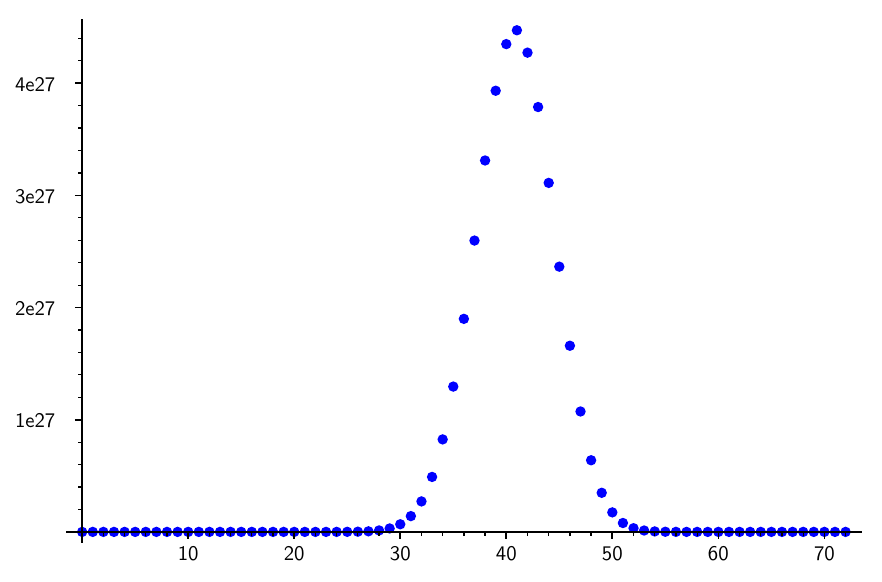}
    \caption{Number of coherent paths on $\HypSimplTwo[50]$ for length $\ell\in[3,73]$}
    \label{fig:NumberPathsByLength}
\end{figure}

Thanks to \Cref{prop:RecursiveFormulaLength}, we can compute these sequences for large $n$, \eg $n = 50$ in \Cref{fig:NumberPathsByLength}.
All out computations tend to confirm this conjecture: it holds for $n\leq 150$.
Besides, note that the archetypal sequence $\bigl(\binom{n}{\ell} ~;~ \ell\in[0,n]\bigr)$ is log-concave and shares a property similar to \Cref{thm:PolynomialCoefficient}: for a fixed $\ell$, the value $\binom{n}{\ell}$ is a polynomial in $n$ of degree $\ell$.
We did not yet prove this conjecture, but there may be a way to tackle it with the matrix recursion of \Cref{prop:RecursiveFormulaLength}.

Moreover, one can count the number of monotone path according to their length (without restricting to coherent monotone paths), which amounts to counting the total number of diagonal-avoiding lattice paths.
Given the type of combinatorics at stake, it seems reasonable (yet unproven) to think that the sequence of total number of monotone paths will also be log-concave.

Note that the unimodality of the number of (coherent) monotone paths per length does not hold for all polytopes.
An article exploring examples and counter-examples of this property is in preparation with Martina Juhnke.

\vspace{0,15cm}

\paragraph{Monotone path polytopes of $\HypSimpl$}
\Cref{thm:EnhancedStepsCriterion} gives a necessary criterion for a monotone path on $\HypSimpl$ to be coherent.
We have shown that this criterion is sufficient for the case $k = 2$, but computer experiments shows that is it no longer sufficient when $k \geq 3$, that is to say there exist monotone paths satisfying the enhanced steps criterion that are not coherent.
The encoding of monotone paths on $\HypSimpl$ through lattice paths on the grid $[n]^k$ seems a good framework for studying this problem further.

\begin{question}
Find a combinatorial criterion for a diagonal-avoiding path in the $[n]^k$ grid to represent a coherent monotone path on $\HypSimpl$ (for a generic direction $\b c$), and count these paths.
\end{question}

\vspace{0,15cm}

\paragraph{Full description of $\MPPHypSimplTwo$}
We only give here a description of the vertices of $\MPPHypSimplTwo$, it would be of prime interest to investigate the (higher-dimensional) faces of it.
In \Cref{rmk:NormalFanOfMn2}, we explained how the details of our proof allow to construct the normal cone of each vertex of $\MPPHypSimplTwo$, and hence to retrieve its whole normal fan.
However, a combinatorial interpretation is still missing.

\begin{question}
Give a combinatorial interpretation for the faces of the monotone path polytope (which are not vertices) of $\HypSimplTwo$ (for a generic direction $\b c$).
\end{question}

A first idea to do so is to introduce a notion of adjacencies between coherent lattice paths in order to describe the edges of $\MPPHypSimplTwo$, but the drawings this notion gives birth to are not easy to interpret.
A second idea would be to use the fact that faces of the hypersimplex are again hypersimplices (of lower dimensions): one could try to “see” $\MPPHypSimpl[n-1]$ inside $\MPPHypSimpl$, and recover (properties of) the face lattice of $\MPPHypSimpl$ from there.
A glimpse of this is depicted in \Cref{fig:MPP42,fig:M(52)}: the 5 octagons appearing in the polytope of the second figure shall be thought of as 5 copies of the octagon on the right of the first figure (but it remains hard to explain where the 16 squares come from, and how the faces fit together).

\vspace{0,25cm}

\paragraph{From generic to non-generic directions}
In \Cref{ssec:NonGenericDirection}, we built a bridge between coherent monotone paths on $\HypSimplTwo$ for generic directions and for directions of the form $\b e_S  := \sum_{i\in S} \b e_i$ for $S\subseteq [n]$ (non-generic directions studied in \cite[Section 8]{BlackSanyal-FlagPolymatroids}).
This raises a more general question:

\begin{question}
Suppose one knows the coherent monotone paths for $(\polytopeP, \b c)$ for all generic $\b c$, what can be deduced for the coherent monotone paths for non-generic $\b c$?
Conversely, how can one determine the case of generic $\b c$ from the case of non-generic $\b c$?
\end{question}

The problem has two sides: first, having $\b c$ non-generic means a (coherent) monotone path cannot use certain edges of $\polytopeP$; second, having $\b c$ non-generic means the start and end vertices of (coherent) monotone paths are not uniquely defined but can be chosen arbitrarily in a face.
However, one can pick $\b c$ non-generic and a $\b c$-monotone path $\c L$, then consider $\b c'$ generic arbitrarily close to $\b c$ (this exists because the set on non-generic $\b c$ is of measure $0$, as it is the union of the co-dimension 1 cones of $\c N_{\polytopeP}$).
If there is a coherent monotone path $\c L'$ for $(\polytopeP, \b c')$, such that $\c L$
is a sub-path of $\c L'$, then $\c L$ is also a coherent monotone path for $(\polytopeP, \b c)$.
Exploring the details of this relation goes beyond the scope of the present paper, but the previous reasoning hints at a deeper link between the generic and the non-generic situation.

Another way to consider the matter is to construct, for each path $\c L = (\b v_1, \dots, \b v_r)$ the set of couples $(\b c, \b \omega)$ which satisfy the inequalities:
$$\inner{\b \omega, \b v_{i+1} - \b v_i}\inner{\b c, \b u - \b v_i} - \inner{\b c, \b v_{i+1} - \b v_i}\inner{\b \omega, \b u - \b v_i} > 0 ~~~~
\begin{array}{l}
\text{for } i\in [r-1] \text{ and } \\
\b u \text{ a } \b c\text{-improving neighbor of } \b v_i
\end{array}$$

This is almost a semi-algebraic set given by inequalities of order 2, if we forget to read ``$\b u$ a $\b c$-improving neighbor of $\b v_i$''.
This set is non-empty if and only if there exists $\b c$ such that $\c L$ is a coherent monotone path for $(\polytopeP, \b c)$.
Therefore, one may probe the combinatorics of the monotone path polytopes of $(\polytopeP, \b c)$ for all $\b c$ at the same time by studying the decomposition of $\R^d\times \R^d$ into semi-algebraic cells by the polynomials $\inner{\b \omega, \b v_{i+1} - \b v_i}\inner{\b c, \b u - \b v_i} - \inner{\b c, \b v_{i+1} - \b v_i}\inner{\b \omega, \b u - \b v_i} = 0$ (for instance, this can be address via a cylindrical algebraic decomposition).
The interplay between this algebraic decomposition and case of non-generic $\b c$ has yet to be explored.

\bibliographystyle{alpha}
\bibliography{Bibliography}
\label{sec:biblio}

\end{document}